\documentclass[12pt,a4paper,oneside]{amsart}

\usepackage{amsthm}
\usepackage{amsmath}
\usepackage{amssymb}
\usepackage{amscd}
\usepackage[utf8]{inputenc}
\usepackage{typearea}
\usepackage{eufrak}
\usepackage{yfonts}
\usepackage{textcomp}
\usepackage{mathrsfs}
\usepackage{hyperref}
\usepackage[draft]{fixme} 
\usepackage{pdfsync}
\usepackage[active]{srcltx}
\usepackage{xypic}
\usepackage{verbatim}

\renewcommand{\bar}[1]{\overline{#1}}

\newcommand{\NN}{\mathbb{N}}

\newcommand{\osh}[1][]{\sigma_{#1}}

\newcommand{\dom}{\operatorname{dom}}
\newcommand{\ran}{\operatorname{ran}}

\newcommand{\ideG}{\mathbf{1}_G}
\newcommand{\ideQ}{\mathbf{1}_Q}

\newcommand{\G}{\mathcal{G}}
\newcommand{\HG}{\mathcal{H}}

\newcommand{\Mtop}{\mathcal{M}}
\newcommand{\Gt}{\mathcal{G}_{tight}}
\newcommand{\Cat}{\Lambda}

\newcommand{\Obj}{\Lambda^\circ}

\newcommand{\Inv}{\mathcal{I}}

\newcommand{\Semi}{\mathcal{S}}
\newcommand{\SemiT}{\mathcal{T}}
\newcommand{\Zig}[1][\Lambda]{\mathcal{Z}_{#1}}
\newcommand{\ZigM}[1][\Lambda]{\mathcal{Z}(#1)}
\newcommand{\Idem}[1][\Semi]{\mathcal{E}(#1)}
\newcommand{\Idemp}{\mathcal{E}}
\newcommand{\Ope}{\mathcal{U}}

\newcommand{\Filt}{\hat{\mathcal{E}}_0}
\newcommand{\FiltT}{\hat{\mathcal{E}}_{tight}}
\newcommand{\UlFilt}{\hat{\mathcal{E}}_{\infty}}

\newcommand{\dmap}{\mathbf{d}}
\newcommand{\tmap}{\mathbf{t}}
\newcommand{\dbar}{\bar{\mathbf{d}}}

\newcommand{\elmap}[2]{
	\tau^{#1}\osh^{#2}}

\newtheorem{lemma}{Lemma}[section]
\newtheorem{corollary}[lemma]{Corollary}
\newtheorem{theorem}[lemma]{Theorem}
\newtheorem{proposition}[lemma]{Proposition}

\theoremstyle{definition}
\newtheorem{definition}[lemma]{Definition}

\newtheorem{example}[lemma]{Example}

\newtheorem{notation}[lemma]{Notation}
\newtheorem{remark}[lemma]{Remark}

\begin{document}

\title[Tight groupoids of LCSC]{The tight groupoid of the inverse semigroups of left cancellative small categories.}

\author{Eduard Ortega}
\address{Department of Mathematical Sciences\\
NTNU\\
NO-7491 Trondheim\\
Norway } \email{Eduardo.Ortega@ntnu.no}

\author{Enrique Pardo}
\address{Departamento de Matem\'aticas, Facultad de Ciencias\\ Universidad de C\'adiz, Campus de
Puerto Real\\ 11510 Puerto Real (C\'adiz)\\ Spain.}
\email{enrique.pardo@uca.es}\urladdr{https://sites.google.com/a/gm.uca.es/enrique-pardo-s-home-page/}


\thanks{The second-named author was partially supported by PAI III grant FQM-298 of the Junta de Andaluc\'{\i}a, and by the DGI-MINECO and European Regional Development Fund, jointly, through grant MTM2017-83487-P}




\maketitle
 
 \begin{abstract}
We fix a path model for the space of filters of the inverse semigroup $\mathcal{S}_\Lambda$ associated to a left cancellative small category $\Lambda$. Then, we compute its tight groupoid, thus giving a representation of its $C^*$-algebra as a (full) groupoid algebra. Using it, we characterize when these algebras are simple. Also, we determine amenability of the tight groupoid under mild, reasonable hypotheses.
 \end{abstract}
 
 \maketitle
 
\section*{Introduction}

In \cite{S1}, Spielberg described a new method of defining $C^*$-algebras from oriented combinatorial data, generalizing the construction of algebras from directed graphs, higher-rank graphs, and (quasi-)ordered groups. To this end, he introduced \emph{categories of paths} --i.e. cancellative small categories with no (nontrivial) inverses-- as a generalization of higher rank graphs, as well as ordered groups.The idea is to start with a suitable combinatorial object and define a $C^*$-algebra directly from what might be termed the generalized symbolic dynamics that it induces. Associated to the underlying symbolic dynamics, he present a natural groupoid derived from this structure. The construction also gives rise to a presentation by generators and relations, tightly related to the groupoid presentation. In \cite{S2} he showed that most of the results hold when relaxing the conditions, so that right cancellation or having no (nontrivial) inverses are taken out of the picture. 

In \cite{BKQS}, B\'edos, Kaliszewski, Quigg and Spielberg use Spielberg's construction to extend the notion of self-similar graph introduced in \cite{EP2} --they termed it as ``Exel-Pardo systems''-- to the context of actions of group (potentially, of groupoids) on left cancellative small categories. To this end, they use a  Zappa-Sz\'ep product construction, and studied the representation theory for the Spielberg algebras of the new left cancellative small category associated to this Zappa-Sz\'ep product.

In the present paper, we study Spielberg construction, using a groupoid approach based in the Exel's tight groupoid construction \cite{E1}. To this end, we study various inverse semigroups associated to a left cancellative small category (see e.g. \cite{DM}), we compute a ``path-like'' model for their tight groupoids, and we study the basic properties of its tight groupoid. Also, we show that the tight groupoid for these inverse semigroups coincide with Spielberg's groupoid \cite{S-LMS}. With this tools at hand, we are able to characterize simplicity for the algebras associated to finitely aligned left cancellative small categories, and in particular in the case of Exel-Pardo systems. Finally, we give, under mild and necessary hypotheses, a characterization of amenability for such groupoid.

The contents of this paper can be summarized as follows: In Section 1 we recall some known facts about small categories and inverse semigroups. In Section 2 we study basic properties of the inverse semigroups $\Semi_\Lambda$ and $\SemiT_\Lambda$ associated to a left cancellative small category $\Lambda$. Section 3 is devoted to study filters on a left cancellative small category and their path models. Section 4 deals with defining actions of $\Semi_\Lambda$ on filter spaces, and we picture their tight groupoids. In Section 5 we show that the tight groupoid of $\Semi_\Lambda$ is isomorphic (as topological groupoid) to the Spielberg's groupoid on $\Lambda$. Groupoid properties characterizing simplicity on the associated algebras are stated in Section 6. Section 7 is centered in analyzing Zappa-Sz\'ep products, introduced in \cite{BKQS} to generalize self-similar graphs of \cite{EP2}, from our particular perspective. We close the paper studying, in Section 8, the amenability of the tight groupoid of Zappa-Sz\'ep products.

\section{Basic facts.}

In this section we collect all the background we need for the rest of the paper. 

\subsection{Small categories}
Given a small category $\Cat$, we will denote by $\Obj$ the class of its objects, and we will identify $\Obj$ with the identity morphisms, so that $\Obj\subseteq \Cat$. Given $\alpha\in\Cat$, we will denote by $s(\alpha):=\dom(\alpha)\in \Obj$ and $r(\alpha):=\ran(\alpha)\in \Obj$. The invertible elements of $\Cat$ are 
$$\Cat^{-1}:=\{\alpha\in\Cat: \exists \beta\in\Cat\text{ such that }\alpha\beta=s(\beta)\}\,.$$
\begin{definition}
	Given a small category $\Cat$, and let $\alpha,\beta,\gamma\in\Cat$:
	\begin{enumerate}
		\item $\Cat$ is \emph{left cancellative} if $\alpha\beta=\alpha\gamma$ then $\beta=\gamma$,
		\item $\Cat$ is \emph{right cancellative} if $\beta\alpha=\gamma\alpha$ then $\beta=\gamma$,
		\item $\Cat$ \emph{has no inverses} if $\alpha\beta=s(\beta)$ then $\alpha=\beta=s(\beta)$.
	\end{enumerate}

A \emph{category of paths} is a small category that is right and left cancellative and has no inverses. 
\end{definition}

Notice that if $\Cat$ is either left or right cancellative, then the only idempotents in $\Cat$ are $\Obj$. Indeed, if $\alpha\alpha=\alpha$, since $\alpha=r(\alpha)\alpha=\alpha s(\alpha)$ we have that $\alpha=s(\alpha)$ or $\alpha=r(\alpha)$. 

\begin{definition}\label{defi1_1_1_2}
	Let $\Cat$ be a small category. Given $\alpha,\beta\in\Cat$, we say that $\beta$ \emph{extends $\alpha$} (equivalently \emph{$\alpha$ is an initial segments of $\beta$}) if there exists $\gamma\in\Lambda$ such that $\beta=\alpha\gamma$. We denote by $[\beta]=\{\alpha\in\Cat:\alpha\text{ is an initial segment of }\beta\}$. We write $\alpha\leq \beta$ if $\alpha\in[\beta]$.
\end{definition}

\begin{lemma}\label{lemma1_1_1_3}
Let $\Cat$ be a small category. Then
\begin{enumerate}
	\item the relation $\leq$ is reflexive and transitive,
	\item if $\Cat$ is left cancellative with no inverses, then $\leq$ is a partial order. 
\end{enumerate}
\end{lemma}
\begin{proof}
(1) Clearly $\alpha=\alpha s(\alpha)$, so $\alpha$ extends itself. If $\beta=\alpha\alpha'$ ($\alpha\leq \beta$) and $\gamma=\beta \beta'$ ($\beta\leq \gamma$), then $\gamma=\alpha\alpha'\beta'$ ($\alpha\leq \gamma$). 

(2) Suppose that $\alpha=\beta\beta'$ ($\beta \leq \alpha$) and $\beta=\alpha\alpha'$ ($\alpha\leq \beta$). Then,
$$\alpha s(\alpha)=\alpha=\beta\beta '=\alpha\alpha'\beta'\,.$$
Thus, by left cancellation we have that $s(\alpha)=\alpha'\beta'$, whence since $\Cat$ has no inverses it follows that $\alpha',\beta'\in\Obj$.
\end{proof}

\begin{lemma}[{\cite[Lemma 2.3]{S2}}]\label{lemma1_1_1_4} 
	Let $\Cat$ be a LCSC (Left cancellative small category), and let $\alpha,\beta\in\Cat$. Then, $\alpha\leq \beta$ and $\beta\leq \alpha$ if and only if $\beta\in\alpha\Cat^{-1}=\{\alpha\gamma:\gamma\in \Cat^{-1}\text{ with }r(\gamma)=s(\alpha)\}$.
\end{lemma}

We denote by $\alpha\approx\beta $ if $\beta \in\alpha\Cat^{-1}$. This is an equivalence relation. 
\begin{lemma}[{\cite[Lemma 2.5(ii)]{S2}}]\label{lemma1_1_1_5} 
	Let $\Cat$ be a LCSC, and let $\alpha,\beta\in\Cat$. Then the following are equivalent:\begin{enumerate}
		\item $\alpha\approx \beta$,
		\item $\alpha\Cat=\beta\Cat$,
		\item $[\alpha]=[\beta]$.
	\end{enumerate}
\end{lemma}

\begin{notation}
	Let $\Cat$ be a LCSC. Given $\alpha,\beta\in\Cat$, we say :
	\begin{enumerate}
		\item $\alpha\Cap \beta$ if and only if $\alpha\Cat\cap\beta \Cat\neq\emptyset$,
		\item $\alpha\perp \beta$ if and only if $\alpha\Cat\cap\beta \Cat=\emptyset$.
	\end{enumerate}
\end{notation}

\begin{definition}
	Let $\Cat$ be a LCSC, and let $F\subset \Cat$. The elements of $\bigcap_{\gamma\in F}\gamma\Cat$ are \emph{the common extensions of $F$}. A common extension $\varepsilon$ of $F$  is \emph{minimal} if for any common extension $\gamma$ with $\varepsilon\in \gamma\Cat$ we have that $\gamma\approx \varepsilon$.
\end{definition}

When $\Cat$ has no inverses, given $F\subseteq \Cat$ and given any minimal common extension $\varepsilon$ of $F$, if $\gamma$ is common extension of  $F$ with $\varepsilon\in \gamma\Cat$ then $\gamma= \varepsilon$. We will denote by 
$$\alpha\vee \beta:=\{\text{the minimal extensions of }\alpha\text{ and }\beta\}\,.$$
Notice that if $\alpha\vee \beta\neq \emptyset$ then $\alpha\Cap \beta$, but the converse fails in general. 

\begin{definition}
	A LCSC $\Lambda$ is \emph{finitely aligned} if for every $\alpha,\beta\in\Lambda$ there exists a finite subset $\Gamma\subset \Lambda$ such that $\alpha\Lambda\cap\beta\Lambda=\bigcup_{\gamma\in \Gamma}\gamma\Lambda$.
\end{definition}

When $\Lambda$ is a finitely aligned LCSC, we can always assume that $\alpha\vee\beta=\Gamma$ where $\Gamma$ is a finite set of minimal common extensions of  $\alpha$ and $\beta$.

\subsection{Inverse semigroups}
\begin{definition}
	A semigroup $\Semi$ is an \emph{inverse semigroup} if for every $s\in\Semi$ there exists a unique $s^*\in\Semi$ such that $s=ss^*s$ and $s^*=s^*ss^*$.
\end{definition}

Equivalently, $\Semi$ is an inverse semigroup if and only if the subsemigroup $\Idem:=\{e\in\Semi:e^2=e\}$ of idempotents of $\Semi$ is commutative \cite[Theorem 1.1.3]{L}.

A \emph{monoid} is a semigroup with unit. We say that a semigroup $\Semi$ has zero if there exists $0\in\Semi$ such that $0s=s0=0$ for every $s\in\Semi$.

\begin{definition}
	Given a set $X$, we define the (symmetric) inverse semigroup on $X$ as
	$$\Inv(X):=\{f:Y\to Z:Y,Z\subseteq X\text{ and }f \text{ is a bijection }\}\,,$$
	endowed with operation 
	$$g\circ f:f^{-1}(\ran(f)\cap \dom(g))\longrightarrow g(\ran(f)\cap \dom(g))\,,$$
	and involution 
	$$f^*:=f^{-1}:\ran(f)\longrightarrow \dom(f)\,.$$
	
	Notice that $\Inv(X)$ has unit $\text{Id}_X:X\to X$ and zero being the empty map $0:\emptyset\to \emptyset$.
\end{definition}

The Wagner-Preston Theorem \cite[Theorem 1.5.1]{L} guarantees that every inverse semigroup is a $*$-subsemigroup of $\Inv(X)$ for some suitable set $X$.

\begin{definition}
	Let $\Semi$ be an inverse semigroup, and let $\Idem$ be its subsemigroup of idempotents. Given $e,f\in\Idem$, we say that $e\leq f$ if and only if $e=ef$. We extend this relation to a partial order as follows: given $s,t\in\Semi$, we say that $s\leq t$ if and only if $s=ss^*t=ts^*s$.
\end{definition}

\begin{definition}
	We say that $s,t\in \Semi$ are compatible, denoted by $s\sim t$, if both $s^*t$ and $st^*$ belong to $\Idem$.
\end{definition}

This concept will be essential to understand various properties. 

\begin{lemma}[{\cite[Lemma 1.4.16]{L}}]\label{lemma1_1_2_5} Let $\Sigma\subseteq \Semi$. If $\bigvee_{\alpha\in \Sigma}\alpha\in\Semi$, then the elements of $\Sigma$ are pairwise compatible.
\end{lemma}

We say that $\Semi$ is \emph{(finitely) distributive} if whenever $\Sigma$ is a (finite) subset of $\Semi$ and $s\in\Semi$, if $\bigvee_{\alpha\in \Sigma}\alpha\in\Semi$ then $\bigvee_{\alpha\in \Sigma} s\alpha\in\Semi$ and $s\left(\bigvee_{\alpha\in \Sigma} \alpha\right)=\bigvee_{\alpha\in \Sigma}s\alpha$.

We will say that $\Semi$ is \emph{(finitely) complete} if for every (finite) subset $\Sigma\subseteq \Semi$ of pairwise compatible elements we have that $\bigvee_{\alpha\in \Sigma}\alpha\in\Semi$.

The symmetric inverse monoid $\Inv(X)$ is complete and distributive \cite[Proposition 1.2.1(i-ii)]{L}. But this property is not necessarily inherited by its inverse subsemigroups. Indeed, the point is that given $\Sigma\subseteq \Semi$ a set of pairwise compatible elements, and $s\in\Semi$:
\begin{enumerate}
	\item Not necessarily $\bigvee_{\alpha\in \Sigma}\alpha \in\Semi$,
	\item even if $\bigvee_{\alpha\in \Sigma} \alpha\in\Semi$, it can happen that $\bigvee_{\alpha\in \Sigma}s\alpha\notin \Semi$.
\end{enumerate}

To understand when $f,g\in\Inv(X)$ are compatible elements, and describe who is $f\vee g\in\Inv(X)$, we address the reader to Lawson's monograph \cite[Proposition 1.2.1]{L}. 

\section{The semigroups $\Semi_\Lambda$ and $\SemiT_\Lambda$}

Given a LCSC $\Lambda$, we will define some inverse semigroups associated to $\Lambda$.

\begin{definition}
	Let $\Lambda$ be a LCSC. For any $\alpha\in\Lambda$, we define two elements of $\Inv(\Lambda)$:\begin{enumerate}
		\item  $\osh^\alpha:\alpha\Lambda\to s(\alpha)\Lambda$ given by $\alpha\beta \mapsto \beta\,,$
		\item  $\tau^\alpha:s(\alpha)\Lambda\to \alpha\Lambda$ given by $\beta \mapsto \alpha\beta\,.$
	\end{enumerate}
\end{definition}

Clearly $\osh^\alpha$ is injective, and since $\Lambda$ is left cancellative so is $\tau^\alpha$.
Moreover, 
$$\osh^\alpha=\osh^\alpha\tau^\alpha\osh^\alpha\qquad\text{and}\qquad \tau^\alpha=\tau^\alpha\osh^\alpha\tau^\alpha\,,$$
for every $\alpha\in\Lambda$.

\begin{definition}
	Given a LCSC $\Lambda$, we define the semigroup
	$$\Semi_\Lambda:=\left\langle \osh^\alpha,\tau^\alpha:\alpha\in \Lambda\right\rangle \,.$$
\end{definition} 

\begin{lemma}\label{lemma1_2_3} Let $\Lambda$ be a LCSC. Then $\Semi_\Lambda$ is an inverse semigroup.
\end{lemma}
\begin{proof}
	It is clear, since $\Inv(\Lambda)$  is an inverse semigroup, and $\Semi_\Lambda\subseteq \Inv(\Lambda)$ is closed under composition and inverses. 
\end{proof}

In order to better understand its structure, we will need to consider finite aligned LCSC. First, we introduce a definition.

\begin{definition}\label{Def:Cond U}
Let $\Lambda$ be a finitely aligned LCSC, and let $s\in \Semi_\Lambda$. We say that a presentation $s=\bigvee_{i=1}^n\tau^{\alpha_i}\sigma^{\beta_i}$ is \emph{irredundant} if for all $1\leq i\ne j\leq n$ we have $\alpha_i\not\in [\alpha_j]$ and $\beta_i\not\in [\beta_j]$.
\end{definition}

\begin{remark}\label{irredundant}
	Let  $s=\bigvee_{i=1}^n\tau^{\alpha_i}\sigma^{\beta_i}\in \Semi_\Lambda$, and suppose that for all $1\leq i\ne j\leq n$ we have $\beta_i\not\in [\beta_j]$. 
Now suppose that there exists $1\leq i\neq j\leq n$ with $\alpha_i\leq \alpha_j$, so there exists $\gamma\in \Lambda$ such that $\alpha_i\gamma=\alpha_j$.	
Then we have that 
$$s(\beta_i\gamma)=\alpha_i\gamma=\alpha_j=s(\beta_j)\,.$$	
But since $s:\bigcup_{i=1}^n \dom(\sigma^{\beta_i})\to \bigcup_{i=1}^n\ran(\tau^{\alpha_i})$ is a bijection, we have that $\beta_i\gamma=\beta_j$, so $\beta_i\leq \beta_j$, a contradiction. Thus, $s$ is irredundant. Similarly it can be proved that $s$ is irredundant if and only if  for all $1\leq i\ne j\leq n$ we have $\alpha_i\not\in [\alpha_j]$.
\end{remark}

\begin{lemma}[{\cite[Lemma 3.3 \& Theorem 6.3]{S1}}]\label{lemma1_2_4}
If $\Lambda$ is a  finite aligned LCSC, then every $f\in \Semi_\Lambda$ is the supremum of a finite family of elements of the form $\elmap{\alpha}{\beta}$ with $\alpha,\beta\in \Lambda$ and $s(\alpha)=s(\beta)$. Moreover, if such a decomposition is irredundant, then is unique (up to permutation).
\end{lemma}

Notice that, given any finite family $\{\alpha_1,\ldots,\alpha_n\}\subset \Lambda$, the elements $\{\elmap{\alpha_i}{\alpha_i}\}_{i=1}^n\subset \Semi_\Lambda$ are pairwise compatible, so that $\bigvee_{i=1}^n\elmap{\alpha_i}{\alpha_i}\in \Inv(\Lambda)$, but not necessarily to $\Semi_\Lambda$. Thus, in order to do some essential arguments we need to consider a new object. 

\begin{definition}
	Let $\Lambda$ be  a finitely aligned LCSC. We define 
	$$\SemiT_\Lambda=\left\lbrace\bigvee_{i=1}^n \elmap{\alpha_i}{\beta_i} : \{\elmap{\alpha_i}{\beta_i}\}_{i=1}^n\subset \Semi_\Lambda\text{ are pairwise compatible}\right\rbrace\,.$$
\end{definition}

Clearly, by  Lemma \ref{lemma1_2_4} $\Semi_\Lambda\subseteq \SemiT_\Lambda\subset \Inv(\Lambda)$. Moreover, by \cite[Proposition 1.4.20 \& Proposition 1.4.17]{L}, $\SemiT_\Lambda$ is closed by composition and inverses, and moreover, is finitely distributive. Thus,
\begin{lemma}
	Let $\Lambda$ be a finitely aligned LCSC. Then, $\SemiT_\Lambda$ is an inverse semigroup containing $\Semi_\Lambda$. Moreover, $\SemiT_\Lambda$ is the smallest finitely complete, finitely distributive, inverse semigroup containing $\Semi_\Lambda$.  
\end{lemma}

Now, we will proceed to understand who are the elements in $\Idem[\SemiT_\Lambda]$ and the order relation.

\begin{lemma}\label{lemma1_2_6}
	Let $\Lambda$ be a finitely aligned LCSC. Then $e\in \Idem[\SemiT_\Lambda]$ if and only if $e=\bigvee_{i=1}^n\elmap{\alpha_i}{\alpha_i}$ for some $\alpha_1,\ldots,\alpha_n\in\Lambda$.
\end{lemma} 
\begin{proof}
	Let  $e\in \Idem[\SemiT_\Lambda]$, then   $e=\bigvee_{i=1}^n\elmap{\alpha_i}{\beta_i}$.  
	By \cite[Proposition 1.4.17]{L}, $e^*=e=\bigvee_{i=1}^n\elmap{\beta_i}{\alpha_i}$, and 
	$$e=e^*e=\bigvee_{i=1}^n\elmap{\beta_i}{\alpha_i}\elmap{\alpha_i}{\beta_i}=\bigvee_{i=1}^n\elmap{\beta_i}{\beta_i}\,.$$
	Thus, $\bigvee_{i=1}^n\elmap{\beta_i}{\beta_i}=\bigvee_{i=1}^n\elmap{\alpha_i}{\beta_i}$. But then given  $1\leq i\leq n$, we have that 
	$$\beta_i=\left( \bigvee_{i=1}^n\elmap{\beta_i}{\beta_i}\right) (\beta_i)=\left( \bigvee_{i=1}^n\elmap{\alpha_i}{\beta_i}\right) (\beta_i)=\alpha_i\,,$$
	as desired.
\end{proof}

\begin{proposition}\label{propo1_2_9}
	Let $\Lambda$ be a finitely aligned LCSC, and let $e=\bigvee_{i=1}^n\elmap{\alpha_i}{\alpha_i}$, $f=\bigvee_{j=1}^m\elmap{\beta_j}{\beta_j}$ be idempotents of either $\Semi_\Lambda$ or $\SemiT_\Lambda$ (written in irredundant form). Then, the following are equivalent:
	\begin{enumerate}
		\item $e\leq f$,
		\item for each $1\leq k\leq n$, there exists $1\leq l\leq m$ such that $\beta_l\leq \alpha_k$.
	\end{enumerate}
\end{proposition}
\begin{proof}
For $(1)$ implies $(2)$, let $e,f\in \Idem[\SemiT_\Lambda]$ with $e\leq f$. Then, $e=\bigvee_{i=1}^n\elmap{\alpha_i}{\alpha_i}$ and $f=\bigvee_{j=1}^m\elmap{\beta_j}{\beta_j}$. Fix any $1\leq k\leq n$. Then,
$\elmap{\alpha_k}{\alpha_k}\leq e\leq f$ if and only if 
$$\elmap{\alpha_k}{\alpha_k}=\elmap{\alpha_k}{\alpha_k} \left( \bigvee_{j=1}^m\elmap{\beta_j}{\beta_j}\right) \,.$$
Since $\SemiT_\Lambda$ is finitely distributive, we have that
$$\elmap{\alpha_k}{\alpha_k}=\bigvee_{j=1}^m\elmap{\alpha_k}{\alpha_k}\elmap{\beta_j}{\beta_j}=\bigvee_{j=1}^m\left( \bigvee_{\varepsilon\in \alpha_k\vee \beta_j}\elmap{\alpha_k\osh^{\alpha_k}(\varepsilon)}{\beta_j\osh^{\beta_j}(\varepsilon)}\right) \,.$$
Without loss of generality, we can assume that the decomposition is irredundant (by using the reduction argument in the proof of \cite[Theorem 6.3]{S1}. By Lemma \ref{lemma1_2_4}, there exist $1\leq l\leq m$ and $\hat{\varepsilon}\in \alpha_k\vee\beta_l$ such that $\elmap{\alpha_k}{\alpha_k}=\elmap{\alpha_k\osh^{\alpha_k}(\hat{\varepsilon})}{\beta_l\osh^{\beta_l}(\hat{\varepsilon})}$, whence $\alpha_k=\alpha_k\osh^{\alpha_k}(\hat{\varepsilon})=\beta_l\osh^{\beta_l}(\hat{\varepsilon})$. Thus, $\beta_l$ is an initial segment of $\alpha_k$ if and only if $\beta_l\leq \alpha_k$ if and only if $\beta_l\in[\alpha_k]$.

For $(2)$ implies $(1)$, if $\beta_l\leq \alpha_k$, then $\elmap{\alpha_k}{\alpha_k}\leq\elmap{\beta_l}{\beta_l}\leq f$. Since this is true for all $1\leq k\leq n$, we have that $e\leq f$, as desired.

Notice that, even we need $\SemiT_\Lambda$ to argue, the conclusion works for $\Semi_\Lambda$ too.
 \end{proof}

By an analog argument, we have the following result, extending Proposition \ref{propo1_2_9} to any couple of elements of $\Semi_\Lambda$.

\begin{proposition}\label{propo1_2_10}
Let $\Lambda$ be a finitely aligned LCSC, and let $s=\bigvee_{i=1}^n\elmap{\alpha_i}{\beta_i}$, $t=\bigvee_{j=1}^m\elmap{\gamma_j}{\delta_j}$ be elements of either $\Semi_\Lambda$ or $\SemiT_\Lambda$ (written in irredundant form). Then, the following are equivalent:
\begin{enumerate}
\item $s\leq t$,
\item for each $1\leq k\leq n$, there exists $1\leq l\leq m$ such that $\alpha_k=\gamma_l \epsilon$ and $\beta_k=\delta_l \epsilon$ for some $\epsilon \in s(\gamma_l)\Lambda$.
\end{enumerate}
\end{proposition}

Now, we will connect $\Semi_\Lambda$ with the semigroups appearing in \cite{DM,S1,S2}.

\begin{definition}
	Let $\Lambda$ be a small category. A \emph{zigzag} is an even tuple of the form 
	$$\xi=(\alpha_1,\beta_1,\alpha_2,\beta_2,\ldots,\alpha_n,\beta_n)$$
	with $\alpha_i,\beta_i\in\Lambda$, $r(\alpha_i)=r(\beta_i)$ for every $1\leq i\leq n$ and $s(\alpha_{i+1})=s(\beta_{i})$ of every $1\leq i<n$. We will denote by $\Zig$ the set of zigzags of $\Lambda$. Given $\xi\in\Zig$, we define $s(\xi)=s(\beta_n)$, $r(\xi)=s(\alpha_1)$ and $\bar{\xi}=(\beta_n,\alpha_n,\ldots,\beta_1,\alpha_1)$.
\end{definition}

Every $\xi\in \Zig$ defines a zigzag map $\varphi_\xi\in \Inv(\Lambda)$ by
$$\varphi_\xi=\osh^{\alpha_1}\tau^\beta_1\cdots \osh^{\alpha_n}\tau^{\beta_n}\,.$$
We will denote $\ZigM=\{\varphi_\xi:\xi\in\Zig \}$.  

\begin{remark}\label{rema1_3_2} $\mbox{ }$

	\begin{enumerate}
		\item For every $\alpha\in\Lambda$ we can define $\xi_\alpha:=(r(\alpha),\alpha)$. Notice that $\varphi_{\xi_\alpha}=\tau^\alpha$ and $\varphi_{\bar{\xi}_\alpha}=\osh^\alpha$.
		\item $\ZigM$ is closed by concatenation, and $\varphi_{\xi_1}\circ\varphi_{\xi_2}=\varphi_{\xi_1\xi_2}$.
		\item For every $\xi\in \Zig$, then $\varphi_{\bar{\xi}}=\varphi_\xi^{-1}$. 
		\end{enumerate}
\end{remark}

Thus,
\begin{lemma}[{\cite[Section 7.2]{BKQS}}]\label{zig_zag}
	If $\Lambda$ is a LCSC, then $\ZigM$ is an inverse semigroup. Moreover, $\ZigM=\Semi_\Lambda$.
\end{lemma}
\begin{proof}
	First part is a consequence of Remark \ref{rema1_3_2}(2-3). For the second part, Remark \ref{rema1_3_2}(1-3) implies that $\Semi_\Lambda\subseteq \ZigM$. On the other side, for every $\xi\in \Zig$ we have that $\varphi_\xi\in \Semi_\Lambda$, so that $\ZigM\subseteq \Semi_\Lambda$.
\end{proof}

Hence, when working with $\ZigM$, we benefit of results in previous sections.

\section{Filters on LCSC}

Let $\Lambda$ be a LCSC. We denote by $\Idemp:=\Idem[\Semi_\Lambda]$ the semilattice of idempotents of the inverse semigroup $\Semi_\Lambda$.

\begin{definition}
	A nonempty subset $\eta$ of $\Idemp$ is a \emph{filter} if:
	\begin{enumerate}
		\item $e\in\eta$, $f\in\Idemp$ and $e\leq f$, then $f\in\eta$,
		\item $e,f\in\eta$ then $ef\in\eta$. 
	\end{enumerate}
\end{definition}

The set of filters of $\Idemp$ is denoted by $\Filt$. We can endow $\Filt$ with a topology, as follows.
\begin{definition}\label{Def:TopologyFilters}
	For any $X,Y\subset \Idemp$ finite subsets, define 
	$$\Ope(X,Y):=\{\eta\in\Filt: X\subseteq \eta\text{ and }Y\cap \eta=\emptyset  \}\,.$$
	Then 
	$$\mathcal{T}_\Idemp=\{\Ope(X,Y): X,Y\subseteq \Idemp \text{ finite}\}\,,$$
	is a basis for a topology of $\Filt$, under which $\Filt$ is Hausdorff and locally compact space (See e.g.\cite{EP}).
	\end{definition}

\begin{definition}
	A filter $\eta\in\Filt$ is an \emph{ultrafilter} if it is not properly contained in another filter. Equivalently, $\eta$ is maximal among the filters, partially ordered by inclusion.
\end{definition}

A useful characterization is the following. 

\begin{lemma}[{\cite[Lemma 12.3]{E1}}]\label{lemma2_1_4} A filter $\eta\in \Filt$ is an ultrafilter if and only if $e\in \Idemp$ and $ef\neq 0$ for every $f\in\eta$ implies $e\in\eta$. 
\end{lemma}

We denote by $\UlFilt$ the subspace of ultrafilters of $\Filt$. Usually, $\UlFilt$ is not closed in $\Filt$.

\begin{definition}\label{Def:ClosureFilters}
	We define $\FiltT$ as the closure of $\UlFilt$ in $\Filt$. A filter in $\FiltT$ is called \emph{tight filter}.
\end{definition}  

In order to characterize tight filters,we need to introduce some known concepts. 

\begin{definition}
	Given $X,Y\subset \Idemp$ finite sets, we define
	$$\Idemp^{X,Y}=\{e\in\Idemp: e\leq x \text{ for every }x\in X\text{ and }ey=0\text{ for every }y\in Y \}\,.$$
\end{definition}

\begin{definition}
	Given a subset $F$ of $\Idemp$, a \emph{outer cover for $F$} is a  subset $Z\subset \Idemp$ such that  for every $f\in F$ there exists $z\in Z$ such that $zf\neq 0$. Moreover, $Z$ is a \emph{cover for $F$} if $Z$ is an outer cover for $F$ with  $Z\subseteq F$. 
\end{definition}

Given an idempotent $e\in \Idemp$, we say that $Z\subseteq \Idemp$ is a \emph{cover for $e$} if $Z$ is a cover for the set $\{f\in \Idemp:f\leq e \}$. 

\begin{lemma}[{\cite[Theorem 12.9]{E1}}]\label{lemma2_1_8}
A filter $\eta\in\Filt$ is tight if and only if for every $X,Y\subset \Idemp$ finite sets and for every finite cover $Z$ of $\Idemp^{X,Y}$, $\eta\in\Ope(X,Y)$ implies $Z\cap \eta\neq \emptyset$.  
\end{lemma}

\subsection{Path models} Viewing  some examples of LCSC, as graphs and $k$-graphs, we are interested in obtain practical models of ultrafilters and tight filters. These models should behave, somehow, as paths in a graph.

To guarantee that every filter has a such a path model, we introduce a restriction on $\Lambda$. 

\begin{definition}
	Let $\Lambda$ be a finitely aligned LCSC. We say that a filter $\eta \in \widehat{\mathcal{E}}_0$ enjoys condition $(*)$ if given $\bigvee_{i=1}^n\elmap{\alpha_i}{\alpha_i}\in\eta$,   then there exists $1\leq j\leq n$ such that $\elmap{\alpha_j}{\alpha_j}\in\eta$.
\end{definition} 

Notice that, if $\Lambda$ is singly aligned, then every filter enjoy condition $(*)$ (see e.g. \cite[Proposition 3.5]{DGKMW}).

Before showing how to construct the path model of $\eta$, let us show that there exist filters where this property always holds. 
\begin{lemma}\label{lemma2_2_2} 
Let $\Lambda$ be a finitely aligned LCSC. Then, any $\eta\in\FiltT$ satisfies condition $(*)$.
\end{lemma}
\begin{proof}
	Let $e\in\eta$. By Lemmas \ref{lemma1_2_4} and \ref{lemma1_2_6}, $e=\bigvee_{i=1}^n \elmap{\alpha_i}{\alpha_i}$. Define $X=\{e\}$, $Y=\emptyset$ and $Z=\{\elmap{\alpha_i}{\alpha_i}\}_{i=1}^n$. Since $e=\bigvee_{i=1}^n\elmap{\alpha_i}{\alpha_i}\geq \elmap{\alpha_j}{\alpha_j}$ for every $1\leq j\leq n$, it is clear that $Z\subset \Idemp^{X,Y}$. Also, if $0\neq f\leq e=\bigvee_{i=1}^n\elmap{\alpha_i}{\alpha_i}$, then $f\elmap{\alpha_j}{\alpha_j}=0$ for every $1\leq j\leq n$ will imply that 
	$$0=\bigvee_{i=1}^nf\elmap{\alpha_i}{\alpha_i}=f\left( \bigvee_{i=1}^n\elmap{\alpha_i}{\alpha_i}\right) =fe=f\,,$$
	a contradiction. Hence, $Z$ is a finite cover of $\Idemp^{X,Y}$. Clearly, $\eta\in\mathcal{U}(X,Y)$.
	
	Thus, $\eta\in \FiltT$ implies $\eta\cap Z\neq\emptyset$ by Lemma \ref{lemma2_1_8}, i.e. there exists $1\leq j\leq n$ such that $\elmap{\alpha_j}{\alpha_j}\in\eta$.
\end{proof}

\begin{corollary}\label{corol2_2_3}
If $\Lambda$ is a finitely aligned LCSC, then every $\eta\in\UlFilt$ satisfies condition $(*)$.
\end{corollary}

Now, we proceed to introduce a set of paths.

\begin{definition}
	Let $\Lambda$ be a finitely aligned LCSC. A nonempty subset $F$ of $\Lambda$ is:\begin{enumerate}
		\item \emph{Hereditary}, if $\alpha\in\Lambda$, $\beta\in F$ and $\alpha\leq \beta$ implies $\alpha\in F$,
		\item \emph{(upwards) directed},  $\alpha,\beta\in F$ implies that there exists $\gamma\in F$ with $\alpha,\beta\leq \gamma$.
	\end{enumerate}
We denote $\Lambda^*$ the set of nonempty, hereditary, directed subsets of $\Lambda$.
 \end{definition}

Notice that, if $F\in\Lambda^*$, there exists a unique $v\in \Lambda^0$ such that $F\subset v\Lambda$. Indeed, given any $\alpha,\beta \in F$, there exists $\gamma\in F$ with $\gamma\geq \alpha,\beta$, i.e., $\gamma=\alpha\hat{\alpha}=\beta\hat{\beta}$ for some $\hat{\alpha},\hat{\beta}\in\Lambda$. Thus, $v=r(\alpha)$ for any $\alpha\in F$ is the desired element.

\begin{definition}
	 Given $\eta\in\Filt$, we define $$\Delta_\eta:=\{\alpha\in\Lambda: \elmap{\alpha}{\alpha}\in\eta\}\,.$$
\end{definition}

\begin{lemma}\label{lemma2_2_6}
	Let $\Lambda$ be a finitely aligned LCSC. For every $\eta\in \Filt$  satisfying condition $(*)$ we have that $\Delta_\eta\in \Lambda^*$.
\end{lemma}
\begin{proof}
	By condition $(*)$, $\Delta_\eta\neq \emptyset$. Set $\alpha\in\Lambda$, $\beta\in\Delta_\eta$ such that $\alpha\leq \beta$. Then $\elmap{\beta}{\beta}\leq \elmap{\alpha}{\alpha}$. Since $\elmap{\beta}{\beta}\in\eta$ and $\eta$ is a filter, we have that $\elmap{\alpha}{\alpha}\in\eta$, whence $\alpha\in\Delta_\eta$.
	
	Finally, suppose $\alpha,\beta\in \Delta_\eta$. Then $\elmap{\alpha}{\alpha},\elmap{\beta}{\beta}\in \eta$, so that $\elmap{\alpha}{\alpha}\elmap{\beta}{\beta}=\bigvee_{\varepsilon\in\alpha\vee\beta}\elmap{\varepsilon}{\varepsilon}\in\eta$. By condition $(*)$ there exists $\delta\in (\alpha\vee \beta)\cap \Delta_\eta$, and since $\elmap{\delta}{\delta}\leq \elmap{\alpha}{\alpha},\elmap{\beta}{\beta}$ we have that $\alpha,\beta\leq \delta\in \Delta_\eta$, as desired.  
\end{proof}

\begin{definition}\label{Def: Cond_Ast}
If $\Lambda$ is a finitely aligned LCSC, we define
$$\widehat{\mathcal{E}}_{*}=\{ \eta \in \Filt : \eta \text{ satisfies condition } (*)\}.$$
\end{definition}

Thus,
\begin{corollary}\label{corol2_2_7}
If $\Lambda$ is a finitely aligned LCSC, then 
$$\begin{array}{rl} \Phi:\widehat{\mathcal{E}}_{*} & \longrightarrow \Lambda^* \\
\eta & \longmapsto \Delta_\eta\,,
\end{array}$$
is a well-defined map.
\end{corollary}
Now, we will construct an inverse for this map.
\begin{definition}
	Given $F\in\Lambda^*$, we define $$\eta_F:=\{f\in\Idemp: f\geq \elmap{\alpha}{\alpha}\text{ for some }\alpha\in F\}\,.$$
\end{definition}

\begin{lemma}\label{lemma2_2_9} If $\Lambda$ is a finitely aligned LCSC, then for every $F\in\Lambda^*$ we have that $\eta_F\in\widehat{\mathcal{E}}_{*}$.
\end{lemma}
\begin{proof}
	Since $F\neq\emptyset$, the set $\{\elmap{\alpha}{\alpha}:\alpha\in F \}\subseteq \eta_F$, whence $\eta_F\neq\emptyset$.  Set $e\in\eta_F$, $f\in \Idemp$ such that $e\leq f$. By hypothesis there exists $\alpha\in F$ such that $\elmap{\alpha}{\alpha}\leq e\leq f$ then $f\in\eta_F$. 
	
	Now set $e,f\in \eta_F$. Then, there exists $\alpha,\beta\in F$ such that $\elmap{\alpha}{\alpha}\leq e$, $\elmap{\beta}{\beta}\leq f$. Since $F$ is directed, there exists $\gamma\in F$ with $\alpha,\beta\leq \gamma$. Thus, $\elmap{\gamma}{\gamma}\leq \elmap{\alpha}{\alpha}\elmap{\beta}{\beta}\leq ef$, whence $ef\in \eta_F$.
	
	Finally, if $f\in \eta_F$, then there exists $\alpha \in F$ such that $f\geq \tau^\alpha \sigma^\alpha$. If $f=\bigvee_{j=1}^m\elmap{\beta_j}{\beta_j}$ (written in irredundant form), by Proposition \ref{propo1_2_9} there exists $1\leq i\leq m$ such that $\tau^{\beta_i}\sigma^{\beta_i}\geq \tau^\alpha \sigma^\alpha$. Since $\eta_F\in \Filt$, we have that $\tau^{\beta_i}\sigma^{\beta_i}\in \eta_F$. Hence, $\eta_F$ satisfies condition $(*)$, so we are done.
\end{proof}

\begin{corollary}\label{corol2_2_10}
	If $\Lambda$ is a finitely aligned LCSC, then 
	$$\begin{array}{rl} \Psi:\Lambda^* & \longrightarrow \widehat{\mathcal{E}}_{*} \\
	F & \longmapsto \eta_F\,,
	\end{array}\,,$$
	is a well-defined map.
\end{corollary}
\begin{lemma}\label{lemma2_2_11}
If $\Lambda$ is a finitely aligned LCSC, then $\Phi$ and $\Psi$ are naturally inverse bijections. 
\end{lemma}
\begin{proof}
	Let $\eta\in\Filt$, and compute 
	\begin{align*}
	\Psi\circ\Phi (\eta)& =\Psi(\Phi(\eta))=\Psi(\Delta_\eta)= \\ 
	& = \{e\in\Idemp:e\geq \elmap{\alpha}{\alpha}\text{ for some }\alpha\in \Delta_\eta \}  \\
	& =\{e\in\Idemp:e\geq \elmap{\alpha}{\alpha}\in\eta \}\,.	 
	\end{align*}
	
	Thus, $\Psi\circ \Phi(\eta)\subseteq \eta$. On the reverse sense, if $e\in \eta$ and $e=\bigvee_{i=1}^n\elmap{\alpha_i}{\alpha_i}$, by condition $(*)$ there exists $1\leq j\leq n$ such that $\elmap{\alpha_j}{\alpha_j}\in\eta$. Thus, $\alpha_j\in \Delta_\eta$ and $e\geq \elmap{\alpha_j}{\alpha_j}$, whence $e\in \Psi\circ\Phi(\eta)$, and so $\Psi\circ\Phi(\eta)\supseteq\eta$. 
	
	Conversely, given $F\in\Lambda^*$, compute
	\begin{align*}
	\Phi\circ\Psi(F) & =\Phi(\Psi(F))=\Psi(\eta_F)=\Delta_{\eta_F} \\
	& = \{\alpha\in \Lambda:\elmap{\alpha}{\alpha}\geq \elmap{\beta}{\beta}\text{ for some }\beta\in F \} \\
	& = \{\alpha\in\Lambda: \alpha\leq \beta\text{ for some }\beta\in F \}\,.
	\end{align*}
Clearly, $\Phi\circ\Psi(F)\subseteq F$. On the reverse sense, if $\alpha\in F$ then $\elmap{\alpha}{\alpha}\in\eta_F$, and thus $\alpha\in\Lambda$ and $\elmap{\alpha}{\alpha}\geq \elmap{\beta}{\beta}$ for some $\beta\in F$, whence $\alpha\in \Phi\circ\Psi(F)$. Thus, $F=\Phi\circ\Psi(F)$. 

\end{proof}

\subsection{Topology of $\Lambda^*$}
Before tracking $\UlFilt$ and $\FiltT$ through $\Psi$, we need to consider a suitable topology defined on $\Lambda^*$.
\begin{definition}\label{top_sets}
	Let $\Lambda$ be a finitely aligned LCSC. Then given $X,Y\subset \Lambda$ finite sets, we define 
	$$\Mtop^{X,Y}=\{F\in\Lambda^*: X\subseteq F\text{ and }Y\cap F=\emptyset \}\,.$$
	We will endow a topology on $\Lambda^*$, with a basis of open sets
	$$\{\Mtop^{X,Y}: X,Y\subseteq \Lambda\text{ finite sets}\}\,.$$
\end{definition}

On the other side, since $\widehat{\mathcal{E}}_{*}\subseteq \Filt$, we can equip $\widehat{\mathcal{E}}_{*}$ with the induced topology. To simplify, we also use $\mathcal{U}(X,Y)$ (see Definition \ref{Def:TopologyFilters}) to denote the basic open sets of the topology for $\widehat{\mathcal{E}}_{*}$. Since both $\widehat{\mathcal{E}}_{\infty}$ and $\widehat{\mathcal{E}}_{\text{tight}}$ are subspaces of $\widehat{\mathcal{E}}_{*}$, in particular the closure of $\widehat{\mathcal{E}}_{\infty}$ in $\Filt$ coincides with the closure of $\widehat{\mathcal{E}}_{\infty}$ in $\widehat{\mathcal{E}}_{*}$. 

We will show that $\Phi$ and $\Psi$ are continuous (and thus homeomorphism) with these topologies on $\widehat{\mathcal{E}}_{*}$ and $\Lambda^*$.
\begin{lemma}\label{lemma2_3_2}
Let $\Lambda$ be a finitely aligned LCSC. Then,
$$\{\Ope(X,Y): X=\{\elmap{\alpha}{\alpha}\}\,,Y=\{\elmap{\beta_i}{\beta_i}\}_{i=1}^n \}\,,$$
is a basis for the topology of $\widehat{\mathcal{E}}_{*}$.
\end{lemma}
\begin{proof}
	Let $e,f_1,\ldots,f_n\in \Idemp$, and consider the basic open set $\Ope(X,Y)$ where $X=\{e\}$ and $Y=\{f_i\}_{i=1}^n$. Set $e=\bigvee_{i=1}^n\elmap{\alpha_i}{\alpha_i}$, $f_i=\bigvee_{j=1}^{m_i}\elmap{\beta_{i,j}}{\beta_{i,j}}$ for $1\leq i \leq m$. Define $\Sigma:=\{\elmap{\beta_{i,j}}{\beta_{i,j}}\}_{i=1,\ldots,m\, j=1,\ldots, m_i}$. Now, given $\eta\in\Ope(X,Y)$, we have that $e\in\eta$  and $\eta\cap Y=\emptyset$.
	
	Since $\eta$ enjoys condition $(*)$, there exists $1\leq j\leq n$ such that $\elmap{\alpha_j}{\alpha_j}\in\eta$. If $\eta\cap \Sigma\neq \emptyset$, then there exists $\elmap{\beta_{i,j}}{\beta_{i,j}}\in\eta$, whence $f_i\geq \elmap{\beta_{i,j}}{\beta_{i,j}}\in\eta$, and thus $\eta\cap Y\neq \emptyset$,  a contradiction.
	Hence, there exists $1\leq j\leq n$ such that $\eta\in \Ope(\{\elmap{\alpha_j}{\alpha_j}\},\Sigma)$. 
	
	Conversely, if there exists $1\leq j \leq n$ such that $\eta\in \Ope(\{\elmap{\alpha_j}{\alpha_j}\},\Sigma)$, then $\elmap{\alpha_j}{\alpha_j}\leq e$, and since  $\elmap{\alpha_j}{\alpha_j}\in\eta$ we have that $e\in\eta$. Also, if $\eta\cap Y\neq\emptyset$, then there exists $1\leq k\leq m$ such that $f_k\in\eta$. By condition $(*)$, there exists $1\leq l\leq m_k$ such that $\elmap{\beta_{k,l}}{\beta_{k,l}}\in\eta$, whence $\eta\cap \Sigma\neq \emptyset$, a contradiction.  Thus, $\eta\in \Ope(X,Y)$, and then $\Ope(X,Y)=\bigcup_{i=1}^n\Ope(\{\elmap{\alpha_i}{\alpha_i}\},\Sigma)$, so we are done. 
\end{proof}

As a consequence, we have 

\begin{lemma}\label{lemma2_3_3} Let  $\Lambda$ be a finitely aligned LCSC. Then $\Phi$ and $\Psi$ are homeomorphisms.
\end{lemma}
\begin{proof}
By Lemma \ref{lemma2_2_11}, both are injections. Since they are mutually inverses, it is enough to show that they are open maps. We will show that for $\Phi$ (the proof for $\Psi$ is analog). We denote by $B_\Lambda=\{\elmap{\alpha}{\alpha}:\alpha\in \Lambda\}\subset \Idemp$. By Lemma \ref{lemma2_3_2}, $\mathcal{T}_\Idemp=\{\Ope(X,Y): X,Y\subset B_\Lambda\text{ finite sets}\}$ is a basis for the topology of $\widehat{\mathcal{E}}_{*}$.  Now, given a finite set  $E\subset B_\Lambda$, we define $\hat{E}=\{\alpha\in \Lambda:\elmap{\alpha}{\alpha}\in E \}\subset \Lambda$. Fix $\Ope(X,Y)\in \mathcal{T}_\Idemp$ for some finite sets $X,Y\subset B_\Lambda$, and compute 
$$\Phi(\Ope(X,Y))=\Phi(\{\eta\in \widehat{\mathcal{E}}_{*}:X\subseteq \eta\text{ and }Y\cap \eta=\emptyset\})\,.$$
Since $\Phi$ is a bijection, $\eta\cap Y=\emptyset$ if and only if $\Delta_\eta\cap \hat{Y}=\emptyset$, and $\elmap{\alpha}{\alpha}\in\eta$ if and only if $\alpha\in \Delta_\eta$, whence 
$$\Phi(\Ope(X,Y))=\{\Delta_\eta\in \Lambda: \hat{X}\subseteq \Delta_\eta\text{ and }\hat{Y}\cap \Delta_\eta=\emptyset\}\,.$$
Since $\Phi$ is a bijection
$$\Phi(\Ope(X,Y))=\{C\in\Lambda^*:\hat{X}\subseteq C\text{ and }\hat{Y}\cap C=\emptyset \}=\Mtop^{\hat{X},\hat{Y}}\,.$$
Thus, $\Phi$ is open, as desired.
\end{proof}

Now, we will identify both $\Phi(\UlFilt)$ and $\Phi(\FiltT)$ in an intrinsic way. To this end, we will use a key result from \cite{S1}.

\begin{lemma}[{\cite[Lemma 7.3]{S1}}]\label{lemma2_3_4} Let $\Lambda$ be a countable finitely aligned LCSC. If $C\subset \Lambda$ is a directed subset and $\beta\in \Lambda$ is such that $\beta\Cap\alpha$ for every $\alpha\in C$, then there exists $\tilde{C}\subset \Lambda$ directed subset such that $\{\beta\}\cup C\subseteq \tilde{C}$. Moreover, if $C$ is hereditary, then so is $\tilde{C}$. 
\end{lemma}

\begin{definition} Given $\Lambda$ a LCSC, we say that $C\in\Lambda^*$ is \emph{maximal} if whenever $C\subset D$ with $D\in \Lambda^*$ we have that $D=\Lambda$.  We will denote $\Lambda^{**}:=\{C\in\Lambda^*:C\text{ is maximal}\}$.
\end{definition}

\begin{lemma}\label{lemma2_3_6} Let $\Lambda$ be a countable, finitely aligned LCSC. Then given $\eta\in \Filt$ the following statements are equivalent:
\begin{enumerate}
\item $\eta\in\UlFilt$,
\item $\Delta_\eta\in \Lambda^{**}$. 
\end{enumerate}
\end{lemma}
\begin{proof}
$(1)\Rightarrow(2)$. First, if $\beta\in \Lambda$ and $\beta\Cap \alpha$ for every $\alpha\in \Delta_\eta$, we will see that $\beta\in \Delta_\eta$. Notice that $\beta\Cap\alpha$ for every $\alpha\in \Delta_\eta$ if and only if $\elmap{\alpha}{\alpha}\elmap{\beta}{\beta}\neq 0$ for every $\alpha\in \Delta_\eta$. Now, let $e=\bigvee_{i=1}^n\elmap{\alpha_i}{\alpha_i}\in \eta$. By Corollary \ref{corol2_2_3}, there exists $1\leq j\leq n$ such that $\elmap{\alpha_j}{\alpha_j}\in\eta$, whence $\alpha_j\in\Delta_\eta$. Thus,
$\elmap{\beta}{\beta}  e\geq \elmap{\beta}{\beta}\elmap{\alpha_j}{\alpha_j}\neq 0$. Since $\eta\in \UlFilt$, Lemma \ref{lemma2_1_4} implies that $\elmap{\beta}{\beta}\in\eta$, and thus $\beta\in\Delta_\eta$. 

Now, suppose $C\in \Lambda^*$ and $\Delta_\eta\subset C$. If $\beta\in C\setminus \Delta_\eta$, since $C$ is directed we have that $\beta\Cap\alpha$ for every $\alpha\in\Delta_\eta$, whence by  Lemma \ref{lemma2_1_4} $\beta\in \Delta_\eta$. Thus, $\Delta_\eta$ is maximal. 

$(2)\Rightarrow (1)$. Set $F\in \Lambda^{**}$, and take $\eta_F\in \widehat{\mathcal{E}}_{*}$. First, pick $\beta\in \Lambda$ such that $\elmap{\beta}{\beta} e\neq 0$ for every $e\in \eta_F$. In particular, $\elmap{\beta}{\beta}\elmap{\alpha}{\alpha}\neq 0$ for every $\alpha\in F$, whence $\beta\Cap\alpha$ for every $\alpha\in F$. By Lemma \ref{lemma2_3_4} there exists $\tilde{F}\in \Lambda^*$ such that $F\cup\{\beta\}\subset \tilde{F}$. Since $F$ is maximal, $\beta\in F$, and thus $\elmap{\beta}{\beta}\in \eta_F$. Now, let $f\in \Idemp$ with $f  e\neq 0$ for every $e\in \eta_F$. By Lemmas  \ref{lemma1_2_4} and \ref{lemma1_2_6} we have $f=\bigvee_{i=1}^n\elmap{\beta_i}{\beta_i}$, whence for every $\alpha\in F$
\begin{align*}
0\neq f \elmap{\alpha}{\alpha} & = \left( \bigvee_{i=1}^n\elmap{\beta_i}{\beta_i}\right)   \elmap{\alpha}{\alpha}=\bigvee_{i=1}^n\elmap{\beta_i}{\beta_i}\elmap{\alpha}{\alpha}=\bigvee_{i=1}^n\bigvee_{\varepsilon\in\beta_i\vee \alpha}\elmap{\varepsilon}{\varepsilon}\,.
\end{align*}

Suppose that for each $1\leq i\leq n$ there exists $\alpha_i\in F$ such that $\alpha_i\vee \beta_i=\emptyset$. Define $g:=\elmap{\alpha_1}{\alpha_1}\cdots \elmap{\alpha_n}{\alpha_n}\in \eta_F$. Since $F$ is directed, there exists $\gamma\in F$ with $\{\alpha_1,\ldots,\alpha_n\}\leq \gamma$, whence $\elmap{\gamma}{\gamma}\leq g$.  Thus, 
\begin{align*}
0\neq \elmap{\gamma}{\gamma}  f & \leq g  f=(\elmap{\alpha_1}{\alpha_1}\cdots \elmap{\alpha_n}{\alpha_n}) \left(\bigvee_{i=1}^n\elmap{\beta_i}{\beta_i} \right) \\ & =\bigvee_{i=1}^n(\elmap{\alpha_1}{\alpha_1}\cdots\widehat{\elmap{\alpha_i}{\alpha_i}}\cdots \elmap{\alpha_n}{\alpha_n})\elmap{\alpha_i}{\alpha_i}\elmap{\beta_i}{\beta_i} =0\,,
\end{align*}
a contradiction. 
Thus, there exists $1\leq j\leq n$ such that $\alpha\vee\beta_j\neq\emptyset$ for every $\alpha\in F$. Thus, $\beta_j\in F$ by the previous argument, and hence $\eta_F \ni \elmap{\beta_j}{\beta_j}\leq f$, so that $f\in\eta_F$. Then, $\eta_F\in\UlFilt$, as desired.
\end{proof}

Notice that this means that $\Phi(\UlFilt)=\Lambda^{**}$. Since $\Phi$ is continuous and $\FiltT=\overline{\UlFilt}^{\|\cdot\|_{\widehat{\mathcal{E}}_{*}}}$ we have that
$$\Phi\left( \FiltT\right)  =\Phi\left( \overline{\UlFilt}^{\|\cdot\|_{\widehat{\mathcal{E}}_{*}}}\right)=\overline{\Phi(\UlFilt)}^{\|\cdot\|_{\Lambda^*}}= \overline{\Lambda^{**}}^{\|\cdot\|_{\Lambda^*}}=$$
$$=\{C\in\Lambda:\text{for every finite }X,Y\subset \Lambda\text{ with }C\in \Mtop^{X,Y},\text{ there exists }D\in \Lambda^{**}\cap \Mtop^{X,Y} \}.
$$
Now, we will introduce a couple of definitions for tight hereditary directed subsets of $\Lambda$, and we will show that they are equivalent. 

First definition is just the translation of Lemma \ref{lemma2_1_8} to the context of $\Lambda^*$, we need to recover, and extend, the concept from \cite{S1}.
\begin{definition}\label{definition2_3_7}
Let $\Lambda$ be a LCSC and $\alpha\in \Lambda$. A subset $F\subset r(\alpha)\Lambda$ is \emph{exhaustive with respect to $\alpha$} if for every $\gamma\in \alpha\Lambda$ there exists a $\beta\in F$ with $\beta\Cap\gamma$. We denote $\mathsf{FE}(\alpha)$ the collection of finite sets of $r(\alpha)\Lambda$ that are exhaustive with respect to $\alpha$.
\end{definition}

Notice that exhaustive sets corresponds to covers. 
\begin{definition}\label{definition2_3_8}
Let $\Lambda$ be a LCSC and $v\in\Lambda^0$. Then, $C\in v\Lambda^*$ is \emph{tight} if for every $\alpha\in C$, every $\{\beta_1,\ldots,\beta_n\}\cap C=\emptyset$ and any finite exhaustive set $Z$ of $\alpha\Lambda\setminus \bigcup_{i=1}^n\beta_i\Lambda$ we have that $C\cap Z\neq\emptyset$. We denote by $\Lambda_{tight}$ the set of tight hereditary directed sets.
\end{definition}

Now we introduce the new definition. 
\begin{definition}\label{definition2_3_9}
Let $\Lambda$ be a LCSC and $v\in\Lambda^0$. We say that $C\in v\Lambda^*$ is $\emph{E-tight}$ if  for every $\alpha\in C$  and every finite set $F$ of $\Lambda$ with $C\cap F=\emptyset$,  there exists $D\in\Lambda^{**}$ with $\alpha\in D$ and $D\cap F=\emptyset$. We denote by $\Lambda_{E-t}$ the set of $E$-tight hereditary directed sets.
\end{definition}

 We have that $\Lambda_{E-t}=\Phi(\FiltT)$. Now,
 \begin{lemma}\label{lemma2_3_10}
Let $\Lambda$ be  a LCSC. Then $\Lambda_{E-t}=\Lambda_{tight}$. 
 \end{lemma}
\begin{proof}
First we prove $\Lambda_{E-t}\subseteq\Lambda_{tight}$. Let $C\in \Lambda_{E-t}$ and let $Y=\{y_1,\ldots,y_n\}\subset \Lambda$ be a finite set with $C\cap Y=\emptyset$. Let $\alpha\in C$ and take any finite exhaustive set $Z=\{z_1,\ldots,z_m\}\subset \alpha\Lambda\setminus \bigcup_{i=1}^ny_i\Lambda$. Suppose that $Z\cap C=\emptyset$. By assumption there exists $D\in \Lambda^{**}$ with $\alpha\in D$ and $D\cap \{Y\cup Z\}=\emptyset$. Since $Y\cap D=\emptyset$, $Z\cap D=\emptyset$ and $D$ is maximal, by Lemma \ref{lemma2_3_4} there exist $x_{y_1},\ldots, x_{y_n}\in D$ with $x_{y_i}\bot y_i$ for every $1\leq i\leq n$, and $x_{z_1},\ldots, x_{z_m}\in D$ with $x_{z_j}\bot z_j$ for every $1\leq j\leq m$. Therefore, since $D$ is a directed set, there exists $w\in D$ with $x_{y_i} \leq w$ for $1\leq i\leq n$,   $x_{z_j}\leq w$ for $1\leq j\leq m$ and $\alpha\leq w$. Observe that $w\notin y_i\Lambda$, because otherwise $y_i\Lambda\cap x_{y_i}\Lambda\neq \emptyset$ for every $1\leq i\leq n$, a contradiction. Thus, $w\in \alpha\Lambda\setminus \bigcup_{i=1}^ny_i\Lambda$, and since $Z$ is an exhaustive set of $\alpha\Lambda\setminus \bigcup_{i=1}^ny_i\Lambda$, there exists $z_k\in Z$ such that $w\Cap z_k$. But then $x_{z_k}\Cap z_k$, a contradiction.

Now we will prove $\Lambda_{E-t}\supseteq\Lambda_{tight}$. Let $C\in v\Lambda^*$ be tight, $\alpha\in C$, and let $F=\{\beta_1,\ldots,\beta_n\}\subset \Lambda$ such that $F\cap C=\emptyset$. Let $\Psi(C)=\eta_C=\{e\in\Idemp: \elmap{\gamma}{\gamma}\leq e\text{ for some }\gamma\in C\}$.  Then, $\elmap{\alpha}{\alpha}\in \eta_C$ and $\elmap{\beta_i}{\beta_i}\notin \eta_C$ for every $1\leq i\leq n$, so $\eta_C\in \Ope(X,Y)$ where   $X=\{\elmap{\alpha}{\alpha}\}$ and $Y=\{\elmap{\beta_i}{\beta_i}\}_{i=1}^n$. Let $Z=\{\elmap{\gamma_j}{\gamma_j}\}_{j=1}^m$ be any finite cover of $\Idemp^{X,Y}$, whence $\{\gamma_j\}_{j=1}^m\subseteq \alpha\Lambda \setminus \bigcup_{i=1}^n\beta_i\Lambda$ is a finite exhaustive set. Therefore, by hypothesis, there exists $1\leq k\leq m$ such that $\gamma_k\in C$, and hence $\elmap{\gamma_k}{\gamma_k}\in \eta_C$, so $Z\cap \eta_C\neq\emptyset$. But then, by Lemma \ref{lemma2_1_8}, it follows that $\eta_C$ is a tight filter. Since $\FiltT$ is the closure of $\UlFilt$, there exists $\xi\in \UlFilt$ such that $\xi\in \Ope(X,Y)$. But then $\Phi(\xi)=\Delta_\xi\in \Lambda^{**}$ by Lemma \ref{lemma2_3_6}, with $\alpha\in \Delta_\xi$ and $\Delta_\xi\cap F=\emptyset$, as desired.
\end{proof}

\begin{example}
Let $E=(E^0,E^1,r,s)$ be a directed graph. A finite path $\alpha$ of $E$ of length $n\geq 1$ is a sequence $\alpha_1\cdots \alpha_n$ where $\alpha_i\in E^1$ for $1\leq i\leq n$ such that $s(\alpha_i)=r(\alpha_{i+1})$ for $1\leq i\leq n-1$. Given a path of length $n$ we define $s(\alpha)=s(\alpha_n)$ and $r(\alpha)=r(\alpha_1)$. We denote by $E^n$ be the sets of paths of length $n$. If we define the paths of length $0$ by $E^0$, and  we denote by $E^*=\bigcup_{i=0}^\infty E^i$ the set of all finite paths of $E$.
	An infinite path $\alpha=\alpha_1\alpha_2\cdots$ of $E$ is an infinite sequence of edges $\alpha_i\in E^1$ such that $r(\alpha_{i+1})=s(\alpha_i)$ for every $i\geq 1$. We denote by $E^\infty$ the set of infinite paths. A singular vertex of $E$ is a vertex $v\in E^0$ such that $|r^{-1}(v)|\in\{0,\infty\}$. We denote by $E^0_{sing}$ the set of singular vertices of $E$,  we denote by $E^0_{source}=\{v\in E^0: r^{-1}(v)=\emptyset\}$ and by  $E^0_{inf}=\{v\in E^0:|r^{-1}(v)|=\infty \}$. Thus, $E^0_{sing}=E^0_{source}\cup E^0_{inf}$.

	 Then we define $\Lambda$ to be the singly aligned LCSC given by the set of finite paths $E^*$. Given a path $\alpha\in E^*$ of length $n$ we define $E_\alpha=\{\alpha_1\cdots \alpha_i: 1\leq i \leq n\}$, where $E_v=\{v\}$ for $v\in E^0$. Moreover given an infinite path $\alpha\in E^\infty$ we define $E_\alpha=\{\alpha_1\cdots \alpha_i: i\geq 1\}$.  It is straightforward to prove that 
	 $$\Lambda^*=\bigcup_{\alpha\in E^\infty}E_\alpha\cup \bigcup_{\alpha\in E^*}E_\alpha\,,\qquad\text{and}\qquad\Lambda^{**}=\bigcup_{\alpha\in E^\infty}E_\alpha\cup \bigcup_{\alpha\in E^*,\,r(\alpha)\in E^0_{sink}}E_\alpha\,.$$
	 
	 Now, given $\alpha\in E^*$ of length $n$ with $r(\alpha)\in E^0_{sing}$, let $k\leq n$ and $\beta_1,\ldots,\beta_m\subseteq E^*\setminus E_\alpha$. If $s(\alpha)\in E^0_{source}$, then $E_\alpha$  trivially belongs to $\Lambda_{E-t}$. Suppose that $s(\alpha)\in E^0_{inf}$. Since $s(\alpha)$ is an infinite emitter, we have that there exists $e\in r^{-1}(s(\alpha))$ such that $e$ is not contained in any path of $\beta_1,\ldots,\beta_m$. Now, let $\gamma$ be a path containing $\alpha e$ that is either infinite or $s(\gamma)\in E^0_{source}$. Then, $\alpha_1\cdots \alpha_k\in E_\gamma$ and $E_\gamma \cap\{\beta_1,\ldots,\beta_m,\}=\emptyset$. Thus, $E_\alpha\in \Lambda_{E-t}$. Conversely, given $\alpha\in E^*$ of length $n$ with $s(\alpha)\in E^0\setminus E^0_{sing}$, then $Y=\{\alpha e:e\in r^{-1}(s(\alpha))\}$ is a finite set, and given any $E_\gamma$ containing $\alpha$ must contain $\alpha e$ for some $e\in r^{-1}(s(\alpha))$. Thus, $E_\alpha\notin \Lambda_{E-t}$, and consequently $\Lambda^{**}\subsetneq \Lambda_{\text{tight}}$. 

\end{example}

\section{Actions of $\Semi_\Lambda$ }
\subsection{Basic definitions} We first recall the basic elements about partial actions of inverse semigroups on $\Filt$ and some subspaces of it. For further references see \cite{EP}.

\begin{definition}
	Let $\Semi$ be  an inverse semigroup, let $\Idemp:=\Idem$ be its semilattice of idempotents and let $\Filt$ be the locally compact Hausdorff space of filters on $\Idemp$. Given any $s\in \Semi$ and any $\eta\in \Filt$ with $s^*s\in \eta$, we define
	$$s\cdot \eta:=\{f\in \Idemp: ses^*\leq f\text{ for some }e\in \eta\}\,,$$
	which is a filter containing $ss^*$. This defines a partial action of $\Semi$ on $\Filt$. For each $s\in \Semi$, the domain of $s\cdot$ is
	$$D_{s^*s}:=\{\eta\in \Filt:s^*s\in \eta \}=\Ope(\{s^*s\},\emptyset)\,,$$
	and the range of $s\cdot$ is $D_{ss^*}:=\Ope(\{ss^*\},\emptyset)$. Thus, $s\cdot$ acts by local homeomorphisms. In particular $s\cdot$ is continuous.  
\end{definition}

Since $s\cdot \eta\in \UlFilt$ for every $\eta\in\UlFilt$ \cite[Proposition 3.5]{EP}, we have that $s\cdot \eta\in \FiltT$ for every $\eta\in \FiltT$ \cite[Proposition 12.11]{E1}.   

Now, we specialize to the case of $\Semi_\Lambda$, when $\Lambda$ is a finite aligned LCSC and $\eta \in \widehat{\mathcal{E}}_{*}$, the reason being that we are interested in define an action of $\Semi_\Lambda$ on $\Lambda^*$, using the homeomorphisms defined in the previous section. The essential step to be covered is to show that the action defined on $\Filt$ restricts to $\widehat{\mathcal{E}}_{*}$.

First, we need to prove a couple of results.

\begin{lemma}\label{lemma_compatibility}
Let $\Lambda$ be a finitely aligned LCSC. Let $s=\bigvee_{i=1}^n\elmap{\alpha_i}{\beta_i}\in \Semi_\Lambda$ be irredundant. Then, for any $1\leq i\ne j\leq n$ such that $\beta_i\Lambda \cap \beta_j\Lambda \ne \emptyset$ and for any $\eta \in \beta_i\Lambda \cap \beta_j\Lambda$ we have that $\elmap{\alpha_i}{\beta_i}(\eta)=\elmap{\alpha_j}{\beta_j}(\eta)$.
\end{lemma}
\begin{proof}
Since $s\in \Semi_\Lambda$, $\elmap{\alpha_i}{\beta_i}$ and $\elmap{\alpha_j}{\beta_j}$ are compatible. Thus,
$$\elmap{\alpha_i}{\beta_i}\elmap{\beta_j}{\alpha_j}=\bigvee_{\varepsilon\in \beta_i\vee \beta_j}\elmap{\alpha_i\sigma^{\beta_i}(\varepsilon)}{\alpha_j\sigma^{\beta_j}(\varepsilon)},$$
is an idempotent, whence by Lemma \ref{lemma1_2_6} we have that $\alpha_i\sigma^{\beta_i}(\varepsilon)=\alpha_j\sigma^{\beta_j}(\varepsilon)$ for every $\varepsilon\in \beta_i\vee \beta_j$. Since $\eta \in \beta_i\Lambda \cap \beta_j\Lambda$, there exists $\varepsilon\in \beta_i\vee \beta_j$ such that 
$$\eta=\varepsilon\widehat{\eta}=\beta_i\widehat{\beta_i}\widehat{\eta}=\beta_j\widehat{\beta_j}\widehat{\eta}.$$

Hence,
$$\elmap{\alpha_i}{\beta_i}(\eta)=\alpha_i\sigma^{\beta_i}(\varepsilon)\widehat{\eta}=\alpha_j\sigma^{\beta_j}(\varepsilon)\widehat{\eta}=\elmap{\alpha_j}{\beta_j}(\eta)\,.$$

\end{proof}

\begin{lemma}\label{lemma_compresion}
Let $\Lambda$ be a finitely aligned LCSC. Let $s=\bigvee_{i=1}^n\elmap{\alpha_i}{\beta_i}\in \Semi_\Lambda$ be irredundant. If $e=\bigvee_{i=1}^k\elmap{\beta_i}{\beta_i}\in \Semi_\Lambda$ for any $1\leq k\leq n$, then $se=\bigvee_{i=1}^k\elmap{\alpha_i}{\beta_i}$ in $\Semi_{\Lambda}$.
\end{lemma}
\begin{proof}
 By hypothesis, $\beta_i\nleq \beta_j$ whenever $i\neq j$. Fix $1\leq k\leq n$, and set $t=\bigvee_{i=1}^k\elmap{\alpha_i}{\beta_i}\in \mathcal{T}_\Lambda$. Now, $se\in \Semi_\Lambda$, and we have that 
 $$\dom(se)=\dom(t)=\bigcup_{i=1}^k\beta_i\Lambda\,.$$
 Let us prove that $se=t$ as function in $\Inv(\Lambda)$; if so, then  we conclude $se=t\in \Semi_\Lambda$. We will compute the image of any element in $\dom(se)$. Pick any $1\leq i\leq k$, $1\leq j\leq n$. Then, we have two options:
\begin{enumerate}
\item $\beta_i\Lambda \cap \beta_j\Lambda = \emptyset$: in this case, $\elmap{\beta_i}{\beta_j}=0$, and thus $\elmap{\alpha_i}{\beta_i}\elmap{\beta_j}{\beta_j}=\elmap{\alpha_i}{\beta_i}(\beta_j)=0$.
\item $\beta_i\Lambda \cap \beta_j\Lambda =\bigcup_{\varepsilon \in \beta_i\vee \beta_j}\varepsilon\Lambda$: in this case
$$\elmap{\alpha_i}{\beta_i}\elmap{\beta_j}{\beta_j}=\bigvee_{\varepsilon\in \beta_i\vee \beta_j}\elmap{\alpha_i\sigma^{\beta_i}(\varepsilon)}{\varepsilon}.$$
Thus, given any $\varepsilon\in \beta_i\vee \beta_j$ and any $\delta\in s(\varepsilon)\Lambda$, we have that $\elmap{\alpha_i}{\beta_i}\elmap{\beta_j}{\beta_j}(\varepsilon\delta)=\elmap{\alpha_i}{\beta_i}(\varepsilon\delta)$. 
\end{enumerate}
Applying Lemma \ref{lemma_compatibility}, we conclude that $se=t$, as desired.
\end{proof}

\begin{lemma}\label{Lem:RestrictionOK}
Let $\Lambda$ be a finite aligned LCSC, let $\eta \in \widehat{\mathcal{E}}_{*}$ and $s\in \Semi_\Lambda$. Then, $s\cdot \eta\in \widehat{\mathcal{E}}_{*}$.
\end{lemma}
\begin{proof}
As noticed before, $s\cdot \eta\in \Filt$. Since $s$ acts on $\eta$, we have $s^*s\in \eta$ and $ss^*\in s\cdot \eta$. If $e\in \eta$, then $s^*s  e\in \eta$. Hence, without lost of generality, we can assume that $e=s^*se$, whence $s^*(ses^*)s=e$. Moreover, since $\eta \in \widehat{\mathcal{E}}_{*}$, we can assume that $e=\elmap{\gamma}{\gamma}$ for some $\gamma\in \Lambda$. 

Now, take $s=\bigvee_{i=1}^n\elmap{\alpha_i}{\beta_i}$.  By \cite[Proposition 1.4.17(1)]{L},
$$s^*s=\bigvee_{i=1}^n\elmap{\beta_i}{\beta_i}.$$
Since $e\leq s^*s$, by Proposition \ref{propo1_2_10} there exists $1\leq j\leq n$ such that $\beta_j\leq \gamma$, i.e. $\gamma=\beta_j\widehat{\beta_j}$. By Lemma \ref{lemma_compresion}, $se=\elmap{\alpha_j\widehat{\beta_j}}{\gamma}$, whence $ses^*=se(se)^*=\elmap{\alpha_j\widehat{\beta_j}}{\alpha_j\widehat{\beta_j}}$.

Now, given any idempotent $f=\bigvee_{k=1}^m\elmap{\delta_k}{\delta_k}$, by Proposition \ref{propo1_2_10}  we have that $s\elmap{\gamma}{\gamma}s^*\leq f$ if and only if there exists $1\leq k\leq m$ such that $\delta_k\leq \alpha_j\widehat{\beta_j}\in \Delta_{s\cdot \xi}$. Then, $\delta_k\in \Delta_{s\cdot \xi}$, and thus $\elmap{\delta_k}{\delta_k}\in s\cdot \xi$. Hence, $s\cdot \xi$ satisfies condition $(*)$, as desired.
\end{proof}

By restricting our attention to $\widehat{\mathcal{E}}_{*}$, we will use the notation $D_{s^*s}$ and $D_{ss^*}$ to refer to the domain and range of the action of an element $s\in \Semi_\Lambda$ on $\widehat{\mathcal{E}}_{*}$. Then, given $s\in \Semi_\Lambda$, we can write (in a unique way up to irredundacy) $s=\bigvee_{i=1}^n\elmap{\alpha_i}{\beta_i}$ by Lemma \ref{lemma1_2_4}, so that $s^*s=\bigvee_{i=1}^n\elmap{\beta_i}{\beta_i}$ by Lemma \ref{lemma1_2_6}, and thus $D_{s^*s}\subseteq \bigcup_{i=1}^nD_{\elmap{\beta_i}{\beta_i}}$. On the other side, if $\elmap{\beta_i}{\beta_i}\in\eta$ for some $1\leq i\leq n$, then $\elmap{\beta_i}{\beta_i}\leq s^*s$ implies that $s^*s\in \eta$, whence $\bigcup_{i=1}^nD_{\elmap{\beta_i}{\beta_i}}\subseteq D_{s^*s}$. Thus, $\bigcup_{i=1}^nD_{\elmap{\beta_i}{\beta_i}}= D_{s^*s}$. Analogously, $\bigcup_{i=1}^nD_{\elmap{\alpha_i}{\alpha_i}}= D_{ss^*}$.

\subsection{The partial action on $\Lambda^*$.}
Since we have an homeomorphism $\Psi:\widehat{\mathcal{E}}_{*}\to \Lambda^*$ with inverse $\Psi$ (Lemma \ref{lemma2_3_3}) we can transfer the action of $\Semi_\Lambda$ on $\widehat{\mathcal{E}}_{*}$ to $\Lambda^*$. First we will fix the domain and range.

\begin{definition}
	Let $s=\elmap{\alpha}{\beta}\in \Semi_\Lambda$. Then, we define 
	$$E_\alpha:=E_{ss^*}=\Phi(D_{ss^*})=\Phi(\Ope(\{\alpha\},\emptyset))=\{C\in\Lambda^*: \alpha\in C \}\,$$
	and
$$E_\beta:=E_{s^*s}=\Phi(D_{s^*s})=\Phi(\Ope(\{\beta\},\emptyset))=\{C\in\Lambda^*: \beta\in C \}\,.$$
	Given $s=\bigvee_{i=1}^n\elmap{\alpha_i}{\beta_i}$ we can define 
	$$E_{s^*s}=\bigcup_{i=1}^n E_{\beta_i}=\bigcup_{i=1}^n\Phi(D_{\elmap{\beta_i}{\beta_i}})=\Phi(D_{s^*s})\,.$$
\end{definition}

The sets $E_{s^*s}$ and $E_{ss^*}$ are the natural candidates for being the domain and range of the partial action of $\Semi_\Lambda$ on $\Lambda^*$. 

Next step is to define the action. We will start by defining the action in the particular case of $s=\elmap{\alpha}{\beta}$. In this case, $F\in E_\beta$ if and only if $\beta\in F$. Then we define 
$$\elmap{\alpha}{\beta}\cdot F=\bigcup_{\beta\leq \gamma,\,\gamma\in F}[\alpha\sigma^\beta (\gamma)]\,,$$
where $[\delta]$ is the set of initial segments of $\delta\in\Lambda$, who clearly belong to $\Lambda^*$. Indeed:
\begin{enumerate}
	\item $\elmap{\alpha}{\beta}(\beta)=\alpha$, thus, $\elmap{\alpha}{\beta}\cdot F\neq \emptyset$.
	\item Set $\eta_1,\eta_2\in \elmap{\alpha}{\beta}\cdot F$, this means that there exist $\gamma_1,\gamma_2$ extensions of $\beta$ with $\gamma_1,\gamma_2\in F$ such that $\eta_i\leq \alpha\sigma^\beta(\gamma_i)$ for $i=1,2$. Since $F$ is directed, $\beta\leq \gamma_1,\gamma_2\leq \delta$ for some $\delta\in F$. Thus, 
	$$\elmap{\alpha}{\beta}(\gamma_1),\elmap{\alpha}{\beta}(\gamma_2)\leq \elmap{\alpha}{\beta}(\delta)\,.$$
	\item If $\delta\in \elmap{\alpha}{\beta}\cdot F$ and $\eta\leq \delta$, then there exists $\gamma\geq \beta$ and $\gamma\in F$ such that $\eta\leq \delta \leq \alpha\sigma^\beta(\gamma)$, so that $\eta\in \elmap{\alpha}{\beta}\cdot F$.
	\end{enumerate}

Moreover, $\elmap{\alpha}{\beta}\cdot F\in E_\alpha$. Thus, in this case we have that
$$\elmap{\alpha}{\beta}\cdot :E_\beta\to E_\alpha\,,\qquad F\mapsto \elmap{\alpha}{\beta}\cdot F\,,$$
is a well-defined map. 

Now, set $s=\bigvee_{i=1}^n\elmap{\alpha_i}{\beta_i}\in \Semi_\Lambda$, with 
$$E_{s^*s}=\bigcup_{i=1}^n E_{\beta_i}\qquad\text{and}\qquad E_{ss^*}=\bigcup_{i=1}^n E_{\alpha_i}\,.$$
By Lemma \ref{lemma1_1_2_5}, $\elmap{\alpha_i}{\beta_i}$ and $\elmap{\alpha_j}{\beta_j}$ are compatible for $1\leq i,j\leq n$, and then $\elmap{\alpha_i}{\beta_i}\elmap{\beta_j}{\alpha_j}$ and $\elmap{\beta_i}{\alpha_i}\elmap{\alpha_j}{\beta_j}$ are idempotents in $\Semi_\Lambda$.

For every $1\leq i\leq n$  and every $F\in E_{\beta_i}$, we can define 
$$s\cdot F=\elmap{\alpha_i}{\beta_i}\cdot F\,.$$

The only point to be checked is that, if $F\in E_{\beta_i}\cap E_{\beta_j}$ for $1\leq i\ne j\leq n$, then $\elmap{\alpha_i}{\beta_i}\cdot F=\elmap{\alpha_j}{\beta_j}\cdot F$. To check this observe that, if $F\in E_{\beta_i}\cap E_{\beta_j}$, then we have that $\beta_i,\beta_j\in F\in \Lambda^*$. Thus, there exists $\gamma\in F$ with $\beta_i,\beta_j\leq \gamma$. Without lost of generality we can assume that $\gamma=\beta_i\vee \beta_j$. Then, $\gamma =\beta_i\sigma^{\beta_i}(\gamma)=\beta_j\sigma^{\beta_j}(\gamma)$. But 
$$\elmap{\alpha_i}{\beta_i}\elmap{\beta_j}{\alpha_j}=\bigvee_{\varepsilon\in \beta_i\vee \beta_j}\elmap{\alpha_i\sigma^{\beta_i}(\varepsilon)}{\alpha_j\sigma^{\beta_j}(\varepsilon)}\in \Idem[\Semi_\Lambda]\,,$$
so that for every $\varepsilon\in \beta_i\vee \beta_j$ we have that $\alpha_i\sigma^{\beta_i}(\varepsilon)=\alpha_j\sigma^{\beta_j}(\varepsilon)$. In particular, for the $\gamma$ above, we have that $\alpha_i\sigma^{\beta_i}(\gamma)=\alpha_j\sigma^{\beta_j}(\gamma)$. Hence,
$$\bigcup_{\beta_i,\beta_j\leq \gamma,\,\gamma\in F}[\alpha_i\sigma^{\beta_i}(\gamma)]=\bigcup_{\beta_i,\beta_j\leq \gamma,\,\gamma\in F}[\alpha_j\sigma^{\beta_j}(\gamma)]\,,$$
that is $\elmap{\alpha_i}{\beta_i}\cdot F=\elmap{\alpha_j}{\beta_j}\cdot F$, as desired. 

Because of this fact, we can define the map
$$s\cdot :E_{s^*s}\to E_{ss^*}\,,$$
as follows: given $F\in E_{s^*s}=\bigcup_{i=1}^n E_{\beta_i}$ we can assume, after re-indexing, that  $\beta_1,\ldots,\beta_k\in F$ and $\beta_{k+1},\ldots,\beta_{n}\notin F$. Thus,
$$s\cdot F=\bigcup_{\bigvee_{i=1}^k\beta_i\leq \gamma,\, \gamma \in F}[\alpha_i\sigma^{\beta_i}(\gamma)]\,.$$

\subsection{The Key Lemma.}
Now, we will prove a result, essential to fix the dictionary.
\begin{lemma}\label{lemma3_3_1}
Let $\Lambda$ be a finitely aligned LCSC. Then, for every $s\in \Semi_\Lambda$ and any $\eta\in D_{s^*s}\cap \widehat{\mathcal{E}_{*}}$ we have that $s\cdot \Delta_\eta=\Delta_{s\cdot \eta}$.
\end{lemma}
\begin{proof}
	Given any $s\in \Semi_\Lambda$, we have that $s=\bigvee_{i=1}^n\elmap{\alpha_i}{\beta_i}$, so that $s^*s=\bigvee_{i=1}\elmap{\beta_i}{\beta_i}$ and $D_{s^*s}=\bigcup_{i=1}^nD_{\elmap{\beta_i}{\beta_i}}$.
	
	By the previous arguments, we can restrict the action of $s$ on $D_{s^*s}$ to an action of $\hat{s}_i:=\elmap{\alpha_i}{\beta_i}$ on $D_{\elmap{\beta_i}{\beta_i}}$ for each particular filter $\eta\in D_{\elmap{\beta_i}{\beta_i}}$. Hence, we can reduce the question to the case $s=\elmap{\alpha}{\beta}$, $s^*s=\elmap{\beta}{\beta}$ and $\eta\in D_{\elmap{\beta}{\beta}}$.
	
	First observe that
	\begin{align*}s\cdot\eta & =\{f\in \Idemp(\Semi_\Lambda): f\geq \elmap{\alpha}{\beta} e \elmap{\beta}{\alpha}\text{ for some }e\in\eta \}\,.
	\end{align*}
	Since $\eta$ satisfies condition $(*)$, if $e=\bigvee_{i=1}^n\elmap{\gamma_i}{\gamma_i}$, then  there exists $1\leq j\leq n$ such that $\elmap{\gamma_j}{\gamma_j}\in \eta$, whence
	$$s\cdot\eta=\{f\in \Idemp(\Semi_\Lambda): f\geq \elmap{\alpha}{\beta}\elmap{\gamma}{\gamma}\elmap{\beta}{\alpha} \text{ for some }\gamma\in \Delta_\eta\}\,.$$
	Now, $\elmap{\alpha}{\beta}\elmap{\gamma}{\gamma}\elmap{\beta}{\alpha} =\bigvee_{\epsilon\in \beta \vee \gamma} \elmap{\alpha\sigma^\beta(\epsilon)}{\alpha\sigma^\beta(\epsilon)}$. Since $\elmap{\beta}{\beta}\in \eta$, we have that $\beta \in \Delta_\eta$. Hence, there exists $\epsilon\in (\beta\vee \gamma)\cap \Delta_\eta$. Thus, 
	$$\elmap{\alpha}{\beta}\elmap{\epsilon}{\epsilon}\elmap{\beta}{\alpha}=\elmap{\alpha\sigma^\beta(\epsilon)}{\alpha\sigma^\beta(\epsilon)}\leq \elmap{\alpha}{\beta}\elmap{\gamma}{\gamma}\elmap{\beta}{\alpha}$$ 
	for some $\epsilon\in (\beta\vee \gamma)\cap \Delta_\eta$, and thus
	$$s\cdot\eta=\{f\in \Idemp(\Semi_\Lambda): f\geq \elmap{\alpha\sigma^\beta(\gamma)}{\alpha\sigma^\beta(\gamma)} \text{ for some }\gamma\in \Delta_\eta\text{ with }\beta\leq \gamma\}\,.$$
	Given $f=\bigvee_{i=1}^n\elmap{\delta_i}{\delta_i}$, since $s\cdot \eta \in \widehat{\mathcal{E}}_{*}$ by Lemma \ref{Lem:RestrictionOK}, then by  Proposition \ref{propo1_2_9} we have that $f\geq \elmap{\alpha\sigma^{\beta}(\gamma)}{\alpha\sigma^{\beta}(\gamma)}$  if and only if there exists $1\leq k\leq  n$ such that  $\delta_k\leq \alpha\sigma^\beta(\gamma)$ for some $\gamma\in \Delta_\eta$ with $\beta\leq \gamma$. Hence, we have that 
	$$s\cdot\eta=\left\lbrace \bigvee_{i=1}^n\elmap{\delta_i}{\delta_i}\in\Idemp(\Semi_\Lambda):\text{there is }1\leq i\leq n\text{ such that } \delta_i\leq \alpha\sigma^\beta(\gamma)\text{ for some }\gamma\in \Delta_\eta \right\rbrace \,.$$
	Thus, 
	$$\Delta_{s\cdot \eta}=\{\delta\in \Lambda: \delta\leq \alpha\sigma^\beta(\gamma)\text{ for some }\gamma\in \Delta_\eta\text{ with }\beta\leq \gamma \}=\bigcup_{\beta\leq \gamma,\,\gamma\in \Delta_\eta}[\alpha\sigma^\beta(\gamma)]=s\cdot \Delta_\eta\,,$$
	 as desired.
\end{proof}

\begin{corollary}\label{corol3_3_2}
Let $\Lambda$ be a countable, finite aligned  LCSC. Then given $s\in \Semi_\Lambda$:
\begin{enumerate}
\item $s\cdot$ restricts to an action on $\Lambda^{**}$,
\item $s\cdot$  restricts to an action on $\Lambda_{tight}$.
\end{enumerate}
\end{corollary}
\begin{proof}
Lemma \ref{lemma3_3_1} shows that for every $s\in \Semi_\Lambda$ and for every $\eta\in D_{s^*s}$ we have that $s\cdot\Phi(\eta)=\Phi(s\cdot \eta)$.

For $(1)$ since $\Phi(\UlFilt)=\Lambda^{**}$, for any $F\in\Lambda^{**}$ we have that $\eta_F=\Psi(F)\in\UlFilt$, so that
$$s\cdot F=s\cdot \Delta_{\eta_F }=s\cdot \Phi(\eta_F)=\Phi(s\cdot \eta_F)\in \Phi(\UlFilt)=\Lambda^{**}\,.$$ 
For $(2)$ since $s\cdot $ is continuous and $\Lambda_{tight}=\overline{\Lambda^{**}}^{\|\cdot\|_{\Lambda^*}}$ the result derives from $(1)$.
\end{proof}

\subsection{The tight groupoid.}
Given an action of an inverse semigroup $\Semi$ on a locally compact Hausdorff space $X$, we can associate to it a groupoid as follows: Consider $\Semi\times X:=\{(s,x):x\in D_{s^*s} \}$ with
\begin{enumerate}
\item $d(s,x)=x$ and $r(s,x)=s\cdot x$,
\item $(s,x)\cdot (t,y)$ is defined if $t\cdot y= x$, and then $(s,x)\cdot (t,y)=(st,y)$,
\item $(s,x)^{-1}=(s^*,s\cdot x)$\,.
\end{enumerate} 
We say that $(s,x)\sim (t,y)$ if and only if $x=y$ and there exists $e\in \Idem$ with $x\in D_e$ and $se=te$. This is an equivalence relation, compatible with the groupoid structure. Thus, we define $\Semi\rtimes X:=\Semi\times X/\sim$, with the induced operations defined above. Moreover, $(\Semi\rtimes X)^{(0)}=X$.
Now to define the topology on $\Semi\rtimes X$, given $s\in \Semi$ and $U\subseteq D_{s^*s}$ an open set, the subset 
$$\Theta(s,U)=\{[s,x]:x\in U\}\,,$$
gives us  a basis for $\Semi\rtimes X$under which it is a locally compact étale groupoid. 

When $X=\FiltT$, $\Semi\rtimes X$ is the tight groupoid  of the inverse semigroup, denote by $\G_{tight}(\Semi)$. For extra information see for example \cite{EP}.

We will show a nice description of $\G_{tight}(\Semi_\Lambda)$.

\begin{lemma}\label{lemma_simpli}
Let $\Lambda$ be a finitely aligned LCSC. Let $s=\bigvee_{i=1}^n\elmap{\alpha_i}{\beta_i}\in\Semi_{\Lambda}$ and let $\xi\in D_{s^*s}$. Suppose that $\beta_k\in \Delta_\xi$ for some $1\leq k\leq n$. Then $[\elmap{\alpha_k}{\beta_k},\xi]=[s,\xi]\in \G_{tight}(\Semi_\Lambda)$.  In particular 
$$\G_{tight}(\Semi_\Lambda)=\{[\elmap{\alpha}{\beta},\xi]: s(\alpha)=\beta,\, \beta\in \Delta_\xi \}\,.$$
\end{lemma}
\begin{proof}
Let $s=\bigvee_{i=1}^n\elmap{\alpha_i}{\beta_i}\in\Semi_{\Lambda}$ and let $[s,\xi]\in \G_{tight}(\Semi_\Lambda)$. Then by definition $[s,\xi]=[se,\xi]$ for any $e\in \xi$. Let $1\leq k\leq n$ such that $\beta_k\in \Delta_\xi$, whence $e=\elmap{\alpha_k}{\beta_k}\in\xi$. Then by Lemma \ref{lemma_compresion} we have that $se=\elmap{\alpha_k}{\beta_k}$. Thus, $[s,\xi]=[\elmap{\alpha_k}{\beta_k}]$. 

Finally, given any $[s,\xi]\in \G_{tight}(\Semi_\Lambda)$ with $s=\bigvee_{i=1}^n\elmap{\alpha_i}{\beta_i}$ and $\xi\in\FiltT$, there exists $\beta_k\in \Delta_\xi$,  since $\xi$ satisfies condition $(*)$. So by the above $[s,\xi]=[\elmap{\alpha_k}{\beta_k},\xi]$.
\end{proof}

Now we are ready to prove the following result.

\begin{lemma}\label{lemma3_4_1}
Let $\Lambda$ be a finite aligned LCSC. Then, $\G_{tight}(\Semi_\Lambda)$ is topologically isomorphic to $\Semi_\Lambda\rtimes \Lambda_{tight}$.
\end{lemma}
\begin{proof}
Since $\Phi:\FiltT\to \Lambda_{tight}$ is a homeomorphism and for every $s\in \Semi_\Lambda$ and  $\eta\in \FiltT$, we have that $s\cdot \Phi(\eta)=\Phi(s\cdot \eta)$, we conclude that the map
$$\rho:\Semi_\Lambda\times \FiltT\to \Semi_\Lambda\times \Lambda_{tight}\qquad \text{given by}\qquad (s,\eta)\mapsto (s,\Phi(\eta))\,,$$
is a groupoid isomorphism. 

Now set $s,t\in \Semi_\Lambda$, $\eta\in D_{s^*s}\cap D_{t^*t}$ and $e\in \Idemp(\Semi_\Lambda)$ with $\eta\in D_e$ and $se=te$. Then, $\Phi(\eta)\in E_{s^*s}\cap E_{t^*t}$ and $\Delta_\eta\in E_e$, so that 
$$(s,\eta)\sim(t,\eta)\qquad \text{if and only if}\qquad (s,\Delta_\eta)\sim (t,\Delta_\eta)\,.$$
Consequently, the isomorphism $\rho$ induces an isomorphism 
$$\hat{\rho}:\Semi_\Lambda\rtimes \FiltT\to \Semi_\Lambda\rtimes \Lambda_{tight}\qquad \text{given by}\qquad [s,\eta]\mapsto [s,\Phi(\eta)]\,.$$ 
Finally, given any $s\in \Semi_\Lambda$ and $U\subseteq D_{s^*s}$ open subset, we have that $\Phi(U)\subseteq E_{s^*s}$ is an open set and $\hat{\rho}(\Theta(s,U))=\Theta(s,\Phi(U))$. Thus, $\hat{\rho}$ is a homeomorphism.
\end{proof}

We are ready to show when this groupoid is Hausdorff. First, we need to recall some known facts.

\begin{definition}
	A poset is a \emph{weak semilattice} if the intersection of principal downsets is finitely generated as a downset.
\end{definition}
In the case of $\Lambda$ being right cancellative, $\Lambda$ can be seen as a subsemigroup of $\Semi_\Lambda$ via the natural map $\alpha\mapsto \tau^{\alpha}$. Hence, when $\Lambda$ is a \textbf{left and right cancellative small category}, we have the following result.

\begin{proposition}[{\cite[Proposition 3.6]{DM}}]\label{propo1_4_2}
	Let $\Lambda$ be a left and right cancellative small category. Then the following are equivalent:
	\begin{enumerate}
		\item $\Lambda$ is finitely aligned,
		\item $\ZigM$ is a weak semilattice. 
	\end{enumerate}
\end{proposition}

An important point is that, when $\Lambda$ fails to be right cancellative, then Donsig \& Millan argument, fails.

\begin{remark}\label{rema1_4_3}
Steinberg \cite[page 1037]{St} says that an inverse semigroup $\mathcal{S}$ is \emph{Hausdorff} if it is a weak semilattice
\end{remark}
The following follows from \cite[Section 5]{St}.

\begin{corollary}\label{corol1_4_4}
Let $\mathcal{S}$ be a countable inverse semigroup. If $\mathcal{S}$ is Hausdorff, then so is $\Gt(\mathcal{S})$.
\end{corollary}

Thus,

\begin{corollary}\label{corollary3_4_2}
If $\Lambda$ is a countable, finite aligned left and right cancellative small category, then $\Gt(\Semi_\Lambda)\cong \Semi_\Lambda\rtimes \Lambda_{tight}$ is Hausdorff.
\end{corollary}
\begin{proof}
The conclusion follows by Lemma \ref{lemma3_4_1} and Corollary \ref{corol1_4_4}.
\end{proof}

If $\Lambda$ fails to be right cancellative, Corollary \ref{corollary3_4_2} would fail in general.

\subsection{Universal tight representations.}

In this subsection we quickly revisit the results proved in \cite{DM}, just fixing the essential hypotheses required to guarantee that these results hold.

First, notice that the results of \cite[Sections 1.1 and 2]{DM} do not require $\Lambda$ to be other than LCSC. In particular, the key result is \cite[Proposition 3.4]{DM}, that works correctly for $\Semi_\Lambda$.

Second, to apply the results of \cite[Section 3]{DM}, we only need to fix the following facts:
\begin{enumerate}
\item As noticed in \cite[Remark before Theorem 10.10]{S2}, the result required to prove \cite[Theorem 3.7]{DM} (namely, \cite[Theorem 8.2]{S1}) do not depend on the amenability of Spielberg's groupoid $G_{\vert \partial\Lambda}$. Thus,
\item Given $\Lambda$ a (countable) finitely aligned LCSC, we define:
\begin{enumerate}
\item $C^*(\Lambda)$ is the universal $C^*$-algebra generated by a family $\{T_\alpha : \alpha \in \Lambda\}$ satisfying:
\begin{enumerate}
\item $T_\alpha^*T_\alpha=T_{s(\alpha)}$.
\item $T_\alpha T_\beta=T_{\alpha \beta}$ if $s(\alpha)=r(\beta)$.
\item $T_\alpha T_\alpha^* T_\beta T_\beta^*=\bigvee_{\gamma\in \alpha \vee \beta} T_\gamma T_\gamma^*$.
\item $T_v=\bigvee_{\alpha \in F} T_\alpha T_\alpha^*$ for every $v\in \Lambda^0$ and for all $F\subset v\Lambda$ finite exhaustive set.
\end{enumerate}
\item Given any unital commutative ring $R$, we define $R\Lambda$ the $R$-algebra generated by a family $\{T_\alpha : \alpha \in \Lambda\}$ satisfying:
\begin{enumerate}
\item $T_\alpha^*T_\alpha=T_{s(\alpha)}$.
\item $T_\alpha T_\beta=T_{\alpha \beta}$ if $s(\alpha)=r(\beta)$.
\item $T_\alpha T_\alpha^* T_\beta T_\beta^*=\bigvee_{\gamma\in \alpha \vee \beta} T_\gamma T_\gamma^*$.
\item $T_v=\bigvee_{\alpha \in F} T_\alpha T_\alpha^*$ for every $v\in \Lambda^0$ and for all $F\subset v\Lambda$ finite exhaustive set.
\end{enumerate}
\end{enumerate}
\end{enumerate}

In order to relate these algebras with the associated tight groupoid, we need to show that the natural representations $\pi : \Semi_\Lambda  \rightarrow  C^*(\Lambda)$ and $\pi :  \Semi_\Lambda  \rightarrow  R\Lambda$ are universal tight. With respect to its tighness, Donsig and Millan \cite[Theorem 3.7]{DM} showed that these representations are cover-to-joint, and the concluded that they are tight. As recently observed by Exel \cite{E2}, that could fail, so that there is an slight imprecision in the proof of \cite[Theorem 2.2]{DM}. Fortunately, Exel solved this problem \cite[Corollary 5.2, Theorem 6.1]{E2}, so that the conclusion remains true. Hence, by \cite[Theorem 10.15]{S2} and \cite[Theorem 3.7]{DM}, we have the following result.

\begin{proposition}\label{Prop:tighness}
Let $\Lambda$ be a (countable) finitely aligned LCSC. Then:
\begin{enumerate}
\item The natural semigroup homomorphism
$$
\begin{array}{cccc}
\pi : & \Semi_\Lambda  & \rightarrow  &  C^*(\Lambda) \\
 & \elmap{\alpha}{\beta} &  \mapsto &   T_\alpha T_\beta^*
\end{array}
$$
is a universal tight representation of $\Semi_\Lambda$ in the category of $C^*$-algebras.
\item For any unital commutative ring $R$, the natural semigroup homomorphism
$$
\begin{array}{cccc}
\pi : & \Semi_\Lambda  & \rightarrow  &  R\Lambda \\
 & \elmap{\alpha}{\beta} &  \mapsto &   T_\alpha T_\beta^*
\end{array}
$$
is a universal tight representation of $\Semi_\Lambda$ in the category of $R$-algebras.
\end{enumerate}
\end{proposition}

Hence, because of \cite[Theorem 13.3]{E1}, Proposition \ref{propo1_4_2} and \cite[Corollary 5.3]{St}, we have

\begin{theorem}\label{Th: GroupoidRep}
Let $\Lambda$ be a (countable) finitely aligned LCSC. Then:
\begin{enumerate}
\item $C^*(\Lambda)\cong C^*(\mathcal{G}_{\text{tight}}(\Semi_\Lambda))$.
\item For any unital commutative ring $R$, $R\Lambda\cong A_R(\mathcal{G}_{\text{tight}}(\Semi_\Lambda))$.
\end{enumerate}
\end{theorem}

\section{Spielberg's groupoid}
In \cite{S-LMS} Spielberg defines a groupoid $\G_{|\partial \Lambda}$ for a category of paths $\Lambda$. We will show that this groupoid is topologically isomorphic to $\G_{tight}(\Semi_\Lambda)$. 

First, Spielberg defines a topology in  $\Lambda^*$ that coincides with the topology we introduce in the definition \ref{top_sets}. Indeed, for $\alpha\in\Lambda$ and $\beta_1,\ldots,\beta_n\in \alpha\Lambda\setminus\{\alpha\}$, and setting $E=\alpha\Lambda\setminus\bigcup_{i=1}^n\beta_i\Lambda$ we have that 
$$\hat{E}=\{C\in\Lambda^*: C\cap \gamma \Lambda \subseteq E \text{ for some }\gamma\in C \}=\Mtop^{\{\alpha\},\{\beta_1,\ldots,\beta_n\}}\,.$$
Therefore, we have that  $\partial \Lambda=\overline{\Lambda^{**}}=\Lambda_{\text{tight}}$. 

Next result is a refined version of \cite[Lemma 4.12]{S1}.

\begin{lemma}\label{lem_eq_sets}
	Let $\Lambda$ be a finitely aligned LCSC. Let $F,G\in\Lambda^*$ and $\alpha,\beta\in\Lambda$ such that $\tau^\alpha\cdot F=\tau^\beta\cdot G$. Then there exists $\delta\in F$ and $\gamma\in G$ such that $\alpha\delta=\beta\gamma$.
\end{lemma}
\begin{proof}
	Let $\delta'\in F$, then $\alpha\delta'\in \tau^\beta\cdot G$. By definition, there exists $\gamma'\in G$ such that $\alpha\delta'\leq \beta\gamma'$. Then, there is $\eta\in\Lambda$ such that $\alpha\delta'\eta=\beta\gamma'$. Now, there exists $\xi\in F$ such that  $\beta\gamma'\leq \alpha\xi$, and hence there is $\eta'\in \Lambda$ such that $\beta\gamma'\eta'=\alpha\xi$, whence $\alpha\delta'\eta\eta'=\alpha\xi$. Now, by left cancellation, we have that $\delta'\eta\eta'=\xi$, so $\delta'\eta\leq\xi\in F$. Then, since $F$ is hereditary, it follows that $\delta'\eta\in F$ too. If we define $\delta:=\delta'\eta$ and $\gamma:=\gamma'$, we are done.
\end{proof}


Now, we recall the definition of Spielberg's groupoid associated to a small category (see e.g. \cite[pp. 729-730]{S-LMS}). We start defining an equivalence relation on $\Lambda \times\Lambda \times \Lambda^*$ by saying that  $(\alpha,\beta, F)\sim (\alpha',\beta',F')$ if there exist $G\in\Lambda^*$, $\gamma,\gamma'\in\Lambda$ such that $F=\tau^\gamma\cdot G$, $F'=\tau^{\gamma'}\cdot G$, $\alpha\gamma=\alpha'\gamma'$ and $\beta\gamma=\beta'\gamma'$. Denote $\mathcal{G}=\Lambda \times\Lambda \times \Lambda^*/\sim$. Now, we define a partial operation on $\mathcal{G}$. To this end, fix the set of composable pairs 
$$\G^{(2)}:=\{ ([\alpha,\beta,F],[\gamma,\delta,G]): \tau^\beta\cdot F=\tau^\gamma\cdot G)  \}\,,$$
and define $[\alpha,\beta, F]^{-1}=[\beta,\gamma, F]$. Given a pair $([\alpha,\beta, F],[\gamma,\delta, G])\in\G^{(2)}$ we define the multiplication by 
$$[\alpha,\beta, F][\gamma,\delta, G]=[\alpha \xi,\delta\eta, H]\,,$$
where $\xi\in F$ and $\eta\in G$ are the elements given in Lemma \ref{lem_eq_sets} such that $\beta \xi=\gamma \eta$, and $H=\sigma^\xi\cdot F=\sigma^\eta\cdot G$. Finally, the sets $[\alpha,\beta, U]:=\{[\alpha,\beta,F]:F\in U\}$ for $U$ an open subset of $\Lambda^*$ forms a basis for the topology of $\G$, under which $\G$ is an \'etale groupoid. By Corollary \ref{corol3_3_2}, we have that $\G_{|\partial \Lambda}=\{[\alpha,\beta,F]\in \G: F\in \Lambda_{tight} \}$.

\begin{proposition}\label{iso_group_sp}
Let $\Lambda$ be a countable, finitely aligned LCSC. Then the map 
$$\Phi : G_{\vert
	\partial\Lambda} \rightarrow \Semi_\Lambda\rtimes \Lambda_{\text{tight}}\,,\qquad [\alpha,\beta, F]\mapsto [\elmap{\alpha}{\beta}, \tau^\beta\cdot F]\,,$$
is an isomorphism of topological groupoids.
\end{proposition}
\begin{proof}
 First, let $(\alpha,\beta, F)\sim (\alpha',\beta',F')$, that this, there exist $G\in\Lambda^*$, $\gamma,\gamma'\in\Lambda$ such that $F=\tau^\gamma\cdot G$, $F'=\tau^{\gamma'}\cdot G$, $\alpha\gamma=\alpha'\gamma'$ and $\beta\gamma=\beta'\gamma'$. Then 
 $$\tau^\beta\cdot F=\tau^\beta\cdot (\tau^\gamma\cdot G)=\tau^{\beta\gamma}\cdot G=\tau^{\beta'\gamma'}\cdot G=\tau^{\beta'}\cdot(\tau^{\gamma'}\cdot G)=\tau^{\beta'}\cdot F'.$$
 Now, $\beta\gamma \in \tau^\beta \cdot F = \tau^{\beta'} \cdot F'$, and 
 $$(\elmap{\alpha}{\beta})(\elmap{\beta\gamma}{\beta\gamma})=\elmap{\alpha\gamma}{\beta\gamma}=\elmap{\alpha'\gamma'}{\beta'\gamma'}=(\elmap{\alpha'}{\beta'})(\elmap{\beta'\gamma'}{\beta'\gamma'})=(\elmap{\alpha'}{\beta'})(\elmap{\beta\gamma}{\beta\gamma}).$$ 
 Hence, $(\tau^\alpha\sigma^\beta, \tau^\beta\cdot F)\sim (\tau^{\alpha'}\sigma^{\beta'}, \tau^{\beta'}\cdot F')$, and thus, $\Phi$ is a well-defined map.
 
 Suppose that $([\alpha, \beta, X], [\gamma, \delta, Y])$ is a composable pair in $\mathcal{G}_{\vert \partial\Lambda}$. Since $\tau^\beta \cdot X=\tau^\gamma\cdot Y$, by \cite[Lemma 4.12]{S1} there exist $\xi, \eta\in \Lambda$, and $Z\in \Lambda_{\text{tight}}$ such that $X=\tau^\xi \cdot Z$, $Y=\tau^\eta\cdot Z$ and $\beta\xi=\gamma\eta$. Then, $\Phi([\alpha, \beta, X])=[\tau^\alpha\sigma^\beta, \tau^\beta \cdot X]$, $\Phi([\gamma, \delta, Y])=[\tau^\gamma\sigma^\delta, \tau^\delta \cdot Y]$, and $\Phi([\alpha, \beta, X][\gamma, \delta, Y])=\Phi([\alpha\xi, \delta\eta, Z])=[\tau^{\alpha\xi}\sigma^{\delta\eta}, \tau^{\delta\eta}\cdot Z]$. Notice that, since $\elmap{\gamma}{\delta}\cdot (\tau^\delta \cdot Y)=\tau^\gamma\cdot Y=\tau^{\gamma\eta}\cdot Z=\tau^{\beta\xi}\cdot Z=\tau^\beta\cdot X$, we can compute
 $$[\tau^\alpha\sigma^\beta, \tau^\beta \cdot X] [\tau^\gamma\sigma^\delta, \tau^\delta \cdot Y]=[\tau^\alpha\sigma^\beta \tau^\gamma\sigma^\delta, \tau^\delta \cdot Y].$$
On one side, $\tau^\delta\cdot Y=\tau^{\delta\eta}\cdot Z$. On the other side, since $\beta\xi=\gamma\eta$, we have $[\tau^\alpha\sigma^\beta, \tau^\beta \cdot X]=[\tau^{\alpha\xi}\sigma^{\beta\xi}, \tau^{\beta\xi} \cdot Z]$, $[\tau^\gamma\sigma^\delta, \tau^\delta \cdot Y]=[\tau^{\gamma\eta}\sigma^{\delta\eta}, \tau^{\delta\eta} \cdot Z]$, and thus
 $$[\tau^\alpha\sigma^\beta, \tau^\beta \cdot X] [\tau^\gamma\sigma^\delta, \tau^\delta \cdot Y]=[\tau^{\alpha\xi}\sigma^{\beta\xi} \tau^{\gamma\eta}\sigma^{\delta\eta}, \tau^{\delta\eta} \cdot Z]=_{(1)}$$
Since $\tau^{\alpha\xi}\sigma^{\beta\xi} \tau^{\gamma\eta}\sigma^{\delta\eta}=\tau^{\alpha\xi}\sigma^{\gamma\eta} \tau^{\gamma\eta}\sigma^{\delta\eta}=\tau^{\alpha\xi}\sigma^{\delta\eta}$, we have
$${ }_{(1)}=[\tau^{\alpha\xi}\sigma^{\delta\eta}, \tau^{\delta\eta} \cdot Z].$$
 So, $\Phi$ is a groupoid homomorphism. 
 
Suppose that $[\alpha, \beta, X], [\gamma, \delta, Y]$ in $\mathcal{G}_{\vert \partial\Lambda}$ such that 
$$\Phi([\alpha, \beta, X])=[\tau^\alpha\sigma^\beta, \tau^\beta \cdot X]=[\tau^\gamma\sigma^\delta, \tau^\delta \cdot Y]=\Phi([\gamma, \delta, Y]).$$
Then, $\tau^\beta\cdot X=\tau^\delta\cdot Y$. By Lemma \ref{lem_eq_sets}, there exist $\xi\in X, \eta\in Y$ such that $\beta\xi=\delta\eta$. Then, the idempotent $e=\tau^{\beta\xi}\sigma^{\beta\xi}=\tau^{\delta\eta}\sigma^{\delta\eta}$ lies in the right domain, and since $[\tau^\alpha\sigma^\beta, \tau^\beta \cdot X]=[\tau^\gamma\sigma^\delta, \tau^\delta \cdot Y]$, left cancellation give us
$\tau^{\alpha\xi}\sigma^{\beta\xi}=\tau^{\alpha}\sigma^{\beta}\cdot e=\tau^{\gamma}\sigma^{\delta}\cdot e=\tau^{\gamma\eta}\sigma^{\gamma\eta}$. Thus, $\tau^{\alpha\xi}=\tau^{\gamma\eta}$, whence $\alpha\xi=\gamma\eta$, and hence $[\alpha, \beta, X]= [\gamma, \delta, Y]$. So, $\Phi$ is injective. 
 
Finally,  let $[\elmap{\alpha}{\beta}, F]\in \Semi_\Lambda\rtimes \Lambda_{\text{tight}}$. Then $\Phi([\alpha,\beta,\sigma^\beta\cdot F])=[\elmap{\alpha}{\beta},\tau^\beta\cdot (\sigma^\beta\cdot F)]=[\alpha,\beta,\elmap{\beta}{\beta}\cdot F]$. But since $\beta\in F$ it follows that $\elmap{\beta}{\beta}\cdot F=F$, so $\Phi$ is exhaustive, and hence $\Phi$ is a topological groupoid isomorphism.  
 
\end{proof}
 
\section{Simplicity}

In \cite[Section 10]{S1} are given conditions in a category of paths $\Lambda$ for $\G_{|\partial \Lambda}$ being topologically free, minimal and locally contractive, but right-cancellation of $\Lambda$ is crucial in the proofs therein. We are going to use  the isomorphism in Proposition \ref{iso_group_sp} and the characterization of these properties given in \cite{EP}, to extend Spielberg results in \cite[Section 10]{S1} to finitely aligned LCSC. 

\begin{definition}
	Let $\Semi$ be an inverse semigroup, and let $s\in \Semi$. Given an idempotent $e\in \Idemp$ such that $e\leq s^*s$, we will say that:
	\begin{enumerate}
		\item $e$ is \emph{fixed} under $s$, if $se=e$,
		\item $e$ is \emph{weakly-fixed} under $s$, if $(sfs^*) f\neq 0$, for every non-zero idempotent $f\leq e$.
	\end{enumerate}
\end{definition}


\begin{definition}
	Given an action $\alpha:\Semi\curvearrowright X$, let $s\in\Semi$, and let $x\in D_{s^*s}$.
	\begin{enumerate}
		\item $\alpha_s(x)=x$, we will say that $x$ is a \emph{fixed point} for $s$. We denote by $F_s$ the set of fix points for $s$.
		\item If there exists $e\in \Idemp$, such that $e\leq s$, and $x\in D_{e}$, we will say that $x$ is a \emph{trivially fixed point} for $s$.
		\item We say that $\alpha$ is a \emph{topologically free} action, if for every $s$ in $\Semi$, the interior of the set of fixed points for $s$ consists of trivial fixed points.  
	\end{enumerate}	 
\end{definition}

Given an action $\alpha:\Semi\curvearrowright X$, the groupoid $\Semi\rtimes X$ is effective if and only if the action $\alpha$ is topologically free \cite[Theorem 4.7]{EP}. 

\begin{remark}\label{remark_top}
	Let $\Lambda$ be a finitely aligned LCSC, and let $\Semi_\Lambda\curvearrowright \FiltT$ be the associated action. Let $s\in\Semi_\Lambda$, and $\xi\in D_{s^*s}\cap \FiltT$. If $s=\bigvee_{i=1}^n\elmap{\alpha_i}{\beta_i}$, then $s^*s=\bigvee_{i=1}^n\elmap{\beta_i}{\beta_i}\in\xi$. Since $\xi$ satisfy condition $(*)$, there exists $1\leq j\leq n$ such that $\elmap{\beta_j}{\beta_j}\in\xi$. Let $C\in \Lambda_{\text{tight}}$ such that $\xi=\eta_C$. Then we have that $\beta_j\in C$. By the definition of the action $\Semi_{\Lambda}\curvearrowright \Lambda_{\text{tight}}$ we have that $s\cdot C=\elmap{\alpha_j}{\beta_j}\cdot C$, and hence $s\cdot \xi=\elmap{\alpha_j}{\beta_j}\cdot \xi$. Thus, without lost of generality, we can assume that $s=\elmap{\alpha_j}{\beta_j}$.
\end{remark}

\begin{theorem}\label{eq_top_free}
	Let $\Lambda$ be a countable, finitely aligned LCSC. If either $\mathcal{G}_{\text{tight}}(\Semi_\Lambda)$ is Hausdorff or $\hat{\mathcal{E}}_\infty =\hat{\mathcal{E}}_{\text{tight}}$, then the following are equivalent:
	\begin{enumerate}
		\item $\G_{tight}(\Semi_\Lambda)$ is effective.
		\item For every $s\in \Semi_\Lambda$, and for every $e\in \Idemp_\Lambda$ which is weakly-fixed under $s$, there exists a finite cover for $e$ consisting of fixed idempotents. 
		\item Given $\alpha,\beta\in \Lambda$ with $r(\alpha)=r(\beta)$ and $s(\alpha)=s(\beta)$, if $\alpha\delta\Cap \beta\delta$ for every $\delta\in s(\alpha)\Lambda$ then there exists $F\in\mathsf{FE}(s(\alpha))$ such that $\alpha\gamma=\beta\gamma$ for every $\gamma\in F$.
	\end{enumerate}
\end{theorem}

\begin{proof}
The equivalence of $(1)$ and $(2)$ follows from Lemma \ref{lemma3_4_1} and Corollary \ref{corollary3_4_2} and \cite[Theorem 4.10, Theorem 3.16 and Theorem 4.7]{EP}.

We assume $(3)$, and we will prove that condition $(iii)$ of \cite[Theorem 4.10]{EP} holds. Let $s=\bigvee_{i=1}^n\elmap{\alpha_i}{\beta_i}\in\Semi_\Lambda$,  and let $\xi\in \UlFilt\cap D_{s^*s}$ with $s\cdot \xi=\xi$ and $\xi\in (F_s)^\circ$. Let $s^*s=\bigvee_{i=1}^n\elmap{\beta_i}{\beta_i}$.  From Remark \ref{remark_top}, there exist $C\in \Lambda^{**}$ such that $\xi=\eta_C$, and $1\leq j\leq n$ such that $\xi\in D_{\elmap{\beta_j}{\beta_j}}$ and $s\cdot \xi=\elmap{\beta_j}{\beta_j}\cdot \xi$. Now, by \cite[Proposition 2.5]{EP} there exists $e=\bigvee_{k=1}^m \elmap{\gamma_k}{\gamma_k}\in \xi$ with $e\leq s^*s$ such that $\xi\in D_e\cap \UlFilt\subseteq(F_s)^\circ$.  Since $\xi$ satisfies condition $(*)$ there exists $1\leq k\leq m$ such that $\elmap{\gamma_k}{\gamma_k}\in\xi$, whence $\xi \in D_{\elmap{\gamma_k}{\gamma_k}}\cap\UlFilt\subseteq (F_s)^\circ$. Then, without lost of generality, we can assume that $s=\elmap{\alpha}{\beta}$, $s^*s=\elmap{\beta}{\beta}$, $\xi=\eta_C$ with $C\in\Lambda^{**}$, $\beta\in C$, and there exists $\gamma\in C$ with $\elmap{\gamma}{\gamma}\leq \elmap{\beta}{\beta}$ such that  $\xi\in D_{\elmap{\gamma}{\gamma}}\subseteq (F_s)^\circ$. Then $\gamma=\beta\hat{\gamma}$ for some $\hat{\gamma}\in\Lambda$, and since by hypothesis $\elmap{\alpha}{\beta}\cdot C=C$, we have that $\elmap{\alpha}{\beta}(\beta\hat{\gamma})=\alpha\hat{\gamma}\in C$ . Now, by \cite[Lemma 4.9]{EP}, $D_{\elmap{\gamma}{\gamma}}\subseteq F_s$ is equivalent to $\elmap{\gamma}{\gamma}$ being weakly fixed under $\elmap{\alpha}{\beta}$. But this means for every $\delta\in s(\hat{\gamma})\Lambda$ we have that
$$\alpha\hat{\gamma}\delta\Cap \beta\hat{\gamma}\delta\,.$$
By hypothesis, there exists $F\in \mathsf{FE}(s(\hat{\gamma}))$ such that $\alpha\hat{\gamma}\delta= \beta\hat{\gamma}\delta$ for every $\delta\in F$. We claim that there exists $\hat{\delta}\in F$ such that $\alpha\hat{\gamma}\hat{\delta}= \beta\hat{\gamma}\hat{\delta}\in C$. Indeed, since $\alpha\hat{\gamma}\in C$, we have that $E:=\sigma^{\alpha\hat{\gamma}}\cdot C\in \Lambda^{**} \subseteq \Lambda_{\text{tight}}$, $\eta_E\in \Ope(\elmap{s(\hat{\gamma})}{s(\hat{\gamma})},\emptyset)$ and $\{\elmap{\delta}{\delta}:\delta\in F \}$ is a cover of $\Ope(\elmap{s(\hat{\gamma})}{s(\hat{\gamma})},\emptyset)$.  Since $\eta_E$ is a tight filter (because $\UlFilt \subseteq \FiltT $) there exists $\hat{\delta}\in F$ with $\elmap{\hat{\delta}}{\hat{\delta}}\in \eta_E$. Then, $\hat{\delta}\in E=\sigma^{\alpha\hat{\gamma}}\cdot C$ and hence $\alpha\hat{\gamma}\hat{\delta}\in C$, as desired.

Now, we define $g:=\elmap{\alpha\hat{\gamma}\hat{\delta}}{\alpha\hat{\gamma}\hat{\delta}}\in \Idemp_\Lambda$. Then, $0\neq \elmap{\alpha\hat{\gamma}\hat{\delta}}{\alpha\hat{\gamma}\hat{\delta}}\leq \elmap{\alpha}{\beta}$ and $\xi=\eta_C\in D_{\elmap{\alpha\hat{\gamma}\hat{\delta}}{\alpha\hat{\gamma}\hat{\delta}}}$. Therefore, $\xi$ is trivially fixed by $s$. Thus, condition $(iii)$ of \cite[Theorem 4.10]{EP} is satisfied, and since either $\G_{tight}(\Semi_\Lambda)$ is Hausdorff or $\hat{\mathcal{E}}_\infty =\hat{\mathcal{E}}_{\text{tight}}$, then condition $(2)$ is satisfied by \cite[Theorem 4.10]{EP}, as desired.	
	
Finally, let us assume $(2)$. Let $\alpha,\beta\in\Lambda$ with $r(\alpha)=r(\beta)$ and $s(\alpha)=s(\beta)$ and satisfying that  $\alpha\delta\Cap \beta\delta$ for every $\delta\in s(\alpha)$, then the idempotent $e=\elmap{\beta}{\beta}$ is weakly fixed under $s:=\elmap{\alpha}{\beta}$, so by hypothesis there exists a finite cover $Z$ of $e$ consisting of fixed idempotents under $s$. Then there exists a finite set  $F\subseteq s(\beta)\Lambda$ such that $Z=\{\elmap{\beta\gamma}{\beta\gamma}\}$. Since the idempotents of $Z$ are fixed under $s$, we have that 
$$\elmap{\alpha}{\beta}\cdot \elmap{\beta\gamma}{\beta\gamma}=\elmap{\alpha\gamma}{\beta\gamma}=\elmap{\beta\gamma}{\beta\gamma}\,,$$
for every $\gamma\in F$. Thus, $\alpha\gamma=\beta\gamma$ for every $\gamma\in F$.
But $Z$ is a cover of $\elmap{\beta}{\beta}$, and hence for every $\delta\in s(\beta)\Lambda$ there exists $\gamma\in F$ such that $\elmap{\beta\delta}{\beta\delta}\cdot \elmap{\beta\gamma}{\beta\gamma}\neq 0$, but this means that $\beta\delta\Cap \beta\gamma$, and hence $\delta\Cap \gamma$ by left-cancellation. Thus, $F\in \mathsf{FE}(s(\beta))$, as desired.  
\end{proof} 

\begin{remark}
Observe that if $\Lambda$ has right cancellation, condition $(3)$ in Theorem \ref{eq_top_free} reduces to aperiodicity as defined in \cite[Definition 10.8]{S1}
\end{remark}
 
\begin{theorem}[{\cite[Theorem 10.14]{S1} \& \cite[Theorem 5.5]{EP}}]\label{eq_minimal}
If $\Lambda$ is a countable, finitely aligned LCSC, then the following statements are equivalent: 
\begin{enumerate}
	\item $\G_{tight}(\Semi_\Lambda)$ is minimal.
	\item For every nonzero $e,f\in \Idemp_\Lambda$, there are $s_1,\ldots,s_n\in \Semi_\Lambda$, such that $\{s_ifs_i^*\}_{i=1}^n$ is an outer cover for $e$.
	\item For every $\alpha,\beta\in \Lambda$ there exists $F\in \mathsf{FE}(\alpha)$ such that for each $\gamma\in F$, $s(\beta)\Lambda s(\gamma)\neq \emptyset$.
\end{enumerate}
\end{theorem}
\begin{proof}
	By \cite[Theorem 5.5]{EP} it is enough to prove the equivalence of conditions $(2)$ and $(3)$. 
	First, we will prove $(3)\Rightarrow(2)$. Without lost of generality, we can  assume that $e=\elmap{\alpha}{\alpha}$ and $f=\elmap{\beta}{\beta}$. By hypothesis there exists $F=\{\gamma_1,\ldots,\gamma_n\}\in \mathsf{FE}(\alpha)$ such that for each $i$ there exists $\delta_i\in s(\beta)\Lambda s(\gamma_i)$. If we define $s_i:=\elmap{\gamma_i}{\beta\delta_i}$ for every $1\leq i\leq n$, we have that $\{s_i fs_i^*\}_{i=1}^n=\{\elmap{\gamma_i}{\gamma_i}\}_{i=1}^n$, that is a cover for $e$.
	
	$(2)\Rightarrow(3)$. Let $e=\elmap{\alpha}{\alpha}$ and $f=\elmap{\beta}{\beta}$. By assumption there exist $s_1,\ldots, s_n$ such that $\{s_ifs^*_i\}_{i=1}^n$ is an outer cover of $e$. Without lost of generality we can assume that $n=1$, so $s:=s_1=\bigvee_{i=1}^m\elmap{\gamma_i}{\delta_i}$, and hence  $sfs^*$ is  an outer cover of $e$. But
	$$sfs^*=\bigvee_{i=1}^m\bigvee_{\varepsilon_i\in \beta \vee \delta_i}\elmap{\gamma_i\sigma^{\delta_i}(\varepsilon_i)}{\gamma_i\sigma^{\delta_i}(\varepsilon_i)}\,.$$
	Since $\Lambda$ is finitely aligned $\beta \vee \delta_i$ is finite for every $1\leq i\leq m$, and so the set $\{\gamma_i\sigma^{\delta_i}(\varepsilon_i): 1\leq i\leq m\,, \varepsilon_i\in \beta\vee \delta_i \}\in \mathsf{FE}(\alpha)$. Finally, observe that since $\varepsilon_i\in \beta\vee \delta_i$ it follows that $s(\beta)\Lambda s(\varepsilon_i)\neq \emptyset$, but $s(\varepsilon_i)=s(\gamma_i\sigma^{\delta_i}(\varepsilon_i))$.
 \end{proof}
 
Then, we have the following result

\begin{theorem}\label{theorem_simpleLambda}
Let $\Lambda$ be a countable, finitely aligned LCSC. If either $\mathcal{G}_{\text{tight}}(\Semi_\Lambda)$ is Hausdorff or $\hat{\mathcal{E}}_\infty =\hat{\mathcal{E}}_{\text{tight}}$, then the following statements are equivalent: 
\begin{enumerate}
\item $C^*(\Lambda)$ is simple.
\item For any field $K$, $K\Lambda$ is simple.
\item The following properties hold:
\begin{enumerate}
\item Given $\alpha,\beta\in \Lambda$ with $r(\alpha)=r(\beta)$ and $s(\alpha)=s(\beta)$, if $\alpha\delta\Cap \beta\delta$ for every $\delta\in s(\alpha)\Lambda$ then there exists $F\in\mathsf{FE}(s(\alpha))$ such that $\alpha\gamma=\beta\gamma$ for every $\gamma\in F$.
\item For every $\alpha,\beta\in \Lambda$ there exists $F\in \mathsf{FE}(\alpha)$ such that for each $\gamma\in F$, $s(\beta)\Lambda s(\gamma)\neq \emptyset$.
\end{enumerate}
\end{enumerate}
\end{theorem}
\begin{proof}
By Theorem \ref{Th: GroupoidRep}, $C^*(\Lambda)\cong C^*(\mathcal{G}_{\text{tight}}(\Semi_\Lambda))$, and for any field $K$, $K\Lambda\cong A_K(\mathcal{G}_{\text{tight}}(\Semi_\Lambda))$. By Theorem \ref{eq_top_free}, condition (2(a)) is equivalent to $\mathcal{G}_{\text{tight}}(\Semi_\Lambda)$ being effective, and by Theorem \ref{eq_minimal}, condition (2(b)) is equivalent to $\mathcal{G}_{\text{tight}}(\Semi_\Lambda)$ being minimal. Then, $(1)\Leftrightarrow (3)$ by \cite[Theorem 5.1]{BCFS}, while $(2)\Leftrightarrow (3)$ by \cite[Theorem 3.5]{St}.
\end{proof}

\section{Zappa-Sz\'ep products of LCSC categories}

In this section we will analyze the notion of Zappa-Sz\'ep products of LCSC categories, introduced in \cite{BKQS}, inspired in the construction of self-similar graphs defined in \cite{EP}. 

Let $\Lambda$ be  a finitely aligned LCSC and let $G$ be a discrete group (with unit $\ideG$). We will use multiplicative notation for the group operation.

We say that the group $G$ acts on $\Lambda$ by permutations when 
$$r(g\cdot\alpha)=g\cdot r(\alpha)\qquad\text{and}\qquad s(g\cdot \alpha)=g\cdot s(\alpha)\qquad\text{for every }\alpha\in\Lambda,\,g\in G\,.$$

For the rest of the section we will assume that $G$ acts by permutations on $\Lambda$.

A \emph{cocyle} for the action of $G$ on $\Lambda$ is a function $\varphi:G\times \Lambda \to G$ satisfying the cocyle identity
$$\varphi(gh,\alpha)=\varphi(g,h\cdot \alpha)\varphi(h,\alpha)\qquad\text{for all }g,h\in G,\,\alpha\in\Lambda\,.$$

In particular the cocycle identity says that $\varphi(\ideG,\alpha)=\ideG$ for every $\alpha\in\Lambda$. 

\begin{definition}\label{definition_CategoryCocycle1}
 A map $\varphi:G\times \Lambda \to \Lambda$ is a \emph{category cocycle} if for all $g\in G$, $v\in\Lambda^0$, and $\alpha,\beta\in\Lambda$ with $s(\alpha)=r(\beta)$ we have 
	\begin{enumerate}
		\item $\varphi(g,v)=g$,
		\item $\varphi(g,\alpha)\cdot r(\alpha)=g\cdot r(\alpha)$,
		\item $g\cdot(\alpha\beta)=(g\cdot \alpha)(\varphi(g,\alpha)\cdot \beta)$,
		\item $\varphi(g,\alpha\beta)=\varphi(\varphi(g,\alpha),\beta)$.
	\end{enumerate}
We call $(\Lambda,G,\varphi)$ a \emph{category system}.
\end{definition}

\begin{definition}\label{definition_zappa-szed1}
	Let $(\Lambda,G,\varphi)$ be a category system. We will denote by $\Lambda\rtimes^\varphi G$ the small category with 
	$$\Lambda\rtimes^\varphi G= \Lambda\times G\qquad\text{and}\qquad(\Lambda\rtimes^\varphi G)^0=\Lambda\times\{e\}\,,$$
	and $r,s:\Lambda\rtimes^\varphi G\to (\Lambda\rtimes^\varphi G)^0$ defined by 
	$$r(\alpha,g)=(r(\alpha),\ideG)\qquad \text{and}\qquad s(\alpha,g)=(g^{-1}\cdot s(\alpha),\ideG)\,.$$
	Moreover for $(\alpha,g),(\beta,h)$ with $s(\alpha,g)=r(\beta,h)$ we have that
	$$(\alpha,g)(\beta,h)=(\alpha(g\cdot\beta),\varphi(g,\beta)h)\,.$$
	We will call $\Lambda\rtimes^\varphi G$ the \emph{Zappa-Sz\'ep product of $(\Lambda,G,\varphi)$}.
\end{definition}

It was proved that $\Lambda\rtimes^\varphi G$ is left cancellative whenever $\Lambda$ is left cancellative \cite[Proposition 3.5]{BKQS}, and as observe in \cite[Remnark 3.9]{BKQS} the elements of the form $(v,g)$ where $v\in \Lambda^0$ and $g\in G$ are units of $\Lambda\rtimes^\varphi G$. Then given $(\alpha,g)\in\Lambda\rtimes^\varphi G$ and $h\in G$ we have that 
$$(\alpha,g)(g^{-1}\cdot s(\alpha), g^{-1}h)=(\alpha,h)\,,$$
so $(\alpha,g)\approx (\alpha,h)$. 
Moreover,  $\Lambda\rtimes^\varphi G$ is finitely aligned (singly aligned) whenever $\Lambda$ is finitely aligned (singly aligned) \cite[Proposition 3.12]{BKQS}. In particular, 
$$(\alpha,g)\vee (\beta,h)=(\alpha\vee \beta)\times\{\ideG\}\,.$$ 

\begin{definition}
A category system $(\Lambda,G,\varphi)$ is called \emph{pseudo free} if, whenever $g\cdot \alpha=\alpha$ and $\varphi(g,\alpha)=\ideG$, then $g=\ideG$. 
\end{definition}

\begin{proposition}[{\cite[Proposition 5.6]{EP2}}]\label{prop_pseudo}
	Let $(\Lambda,G,\varphi)$ be a pseudo free category system. Then, for all $g_1,g_2\in G$, and $\alpha\in\Lambda$, one has that 
	$$g_1\cdot \alpha=g_2\cdot \alpha\text{ and }\varphi(g_1,\alpha)=\varphi(g_2,\alpha)\qquad\Rightarrow\qquad g_1=g_2\,.$$ 
\end{proposition}

\begin{remark}
	Given a $(\Lambda,G,\varphi)$ where $\Lambda$ is a right cancellative category, it may happen that $\Lambda\rtimes^\varphi G$ fails to satisfy right cancellation. Given $(\alpha,a),(\beta,b)$ and $(\gamma,g)$ in $\Lambda\rtimes^\varphi G$ we have that
	$(\alpha,a)(\gamma,g)=(\beta,b)(\gamma,g)$ if and only if $\alpha(a\cdot \gamma)=\beta(b\cdot \gamma)$ and $\varphi(a,\gamma)=\varphi(b,\gamma)$. In particular, the system is  pseudo free  if and only if $\Lambda\rtimes^\varphi G$ is right cancellative.  
	
\end{remark}

\begin{remark} Let $(\Lambda,G,\varphi)$, and let $F=\{(\gamma_1,h_1),\ldots,(\gamma_n,h_n)\}\subseteq \Lambda\rtimes^\varphi G$. Then given $(\alpha,g)\in \Lambda\rtimes^\varphi G$. we have that $F\in \mathsf{FE}(\alpha,g)$ of $\Lambda\rtimes^\varphi G$ if and only if $\{\gamma_1,\ldots,\gamma_n\}\in \mathsf{FE}(\alpha)$ of $\Lambda$. 
\end{remark} 

By the above remark the following results are a direct translation of Theorem \ref{eq_top_free} and \ref{eq_minimal}. 

\begin{proposition}
	Let $(\Lambda,G,\varphi)$ be a category system. If either $\mathcal{G}_{\text{tight}}(\Semi_{\Lambda\rtimes^\varphi  G})$ is Hausdorff or $\hat{\mathcal{E}}_\infty =\hat{\mathcal{E}}_{\text{tight}}$, then the following is equivalent:
	\begin{enumerate}
		\item $\G_{tight}(\Semi_{\Lambda\rtimes^\varphi  G})$ is effective,
		\item Given $(\alpha,a),(\beta,b)\in \Lambda\rtimes^\varphi G$ with $r(\alpha,a)=r(\beta,b)$ and $s(\alpha,a)=s(\beta,b)$, if $(\alpha,a)(\delta,d)\Cap (\beta,b)(\delta,d)$ for every $(\delta,d)\in s((\alpha,a))(\Lambda\rtimes^\varphi G)$ then there exists $F\in\mathsf{FE}(s(\alpha,a))$ such that $(\alpha,a)(\gamma,d)=(\beta,b)(\gamma,d)$ for every $(\gamma,d)\in F$.
		\item Given $\alpha,\beta\in \Lambda$, $a,b\in G$ with $r(\alpha)=r(b)$ and $a^{-1}\cdot s(\alpha)=b^{-1}\cdot s(\beta)$, if $\alpha (a\cdot \delta ) \Cap \beta(b\cdot \delta)$ for every $\delta\in (a^{-1}\cdot s(\alpha))\Lambda$ then there exists $F\in\mathsf{FE}(a^{-1}\cdot s(\alpha))$ such that $\alpha (a\cdot \gamma)=\beta(b\cdot\gamma)$ and $\varphi(a,\gamma)=\varphi(b,\gamma)$ for every $\gamma\in F$.

	\end{enumerate}
\end{proposition}

\begin{proposition} 
		If $(\Lambda,G,\varphi)$ is a category system, then the following statements are equivalent: 
	\begin{enumerate}
		\item $\G_{tight}(\Semi_{\Lambda\rtimes^\varphi  G})$ is minimal.
		\item For every $(\alpha,a),(\beta,b)\in \Lambda\rtimes^\varphi G$ there exists $F\in \mathsf{FE}((\alpha,a))$ such that for each $(\gamma,g)\in F$, $s(\beta,b)(\Lambda\rtimes^\varphi G) s(\gamma,g)\neq \emptyset$.
		\item For every $\alpha,\beta\in \Lambda$ there exists $F\in \mathsf{FE}(\alpha)$ such that for each $\gamma\in F$,  there exist $g\in G$  with $s(\beta)\Lambda  (g\cdot s(\gamma))\neq \emptyset$.

	\end{enumerate}
\end{proposition}

Then, by an analog argument to that of Theorem \ref{theorem_simpleLambda}, we have the following result

\begin{theorem}\label{theorem_simpleSystem}
Let $(\Lambda,G,\varphi)$ be a category system such that $\Lambda$ and $G$ are countable. If either $\mathcal{G}_{\text{tight}}(\Semi_{\Lambda\rtimes^\varphi  G})$ is Hausdorff or $\hat{\mathcal{E}}_\infty =\hat{\mathcal{E}}_{\text{tight}}$, then the following statements are equivalent: 
\begin{enumerate}
\item $C^*(\Semi_{\Lambda\rtimes^\varphi  G})$ is simple.
\item For any field $K$, $K\Semi_{\Lambda\rtimes^\varphi  G}$ is simple.
\item The following properties hold:
\begin{enumerate}
\item Given $\alpha,\beta\in \Lambda$, $a,b\in G$ with $r(\alpha)=r(b)$ and $a^{-1}\cdot s(\alpha)=b^{-1}\cdot s(\beta)$, if $\alpha (a\cdot \delta ) \Cap \beta(b\cdot \delta)$ for every $\delta\in (a^{-1}\cdot s(\alpha))\Lambda$ then there exists $F\in\mathsf{FE}(a^{-1}\cdot s(\alpha))$ such that $\alpha (a\cdot \gamma)=\beta(b\cdot\gamma)$ and $\varphi(a,\gamma)=\varphi(b,\gamma)$ for every $\gamma\in F$.
\item For every $\alpha,\beta\in \Lambda$ there exists $F\in \mathsf{FE}(\alpha)$ such that for each $\gamma\in F$,  there exist $g\in G$  with $s(\beta)\Lambda  (g\cdot s(\gamma))\neq \emptyset$.
\end{enumerate}
\end{enumerate}
\end{theorem}

To end this section, we will have a look on the case of $\Lambda=E^*$, where $E$ is a countable graph. When $G$ is a countable discrete group and $E$ is a countable graph, there is a definition of self-similar graph extending that of \cite{EP2} (see \cite[Definition 2.2]{EPS}). In fact, as shown in \cite[Theorem 3.2]{EPS}, the case of arbitrary graphs can be reduced to the case of row-finite graphs with no sources or sinks up to Morita equivalence (of both algebras and groupoids); in this case, most of the properties enjoyed by the system are analog to these found in the finite case.

In order to fix the relation between $\G_{tight}(\Semi_{G,E})$ and $\G_{tight}(\Semi_{{E^*}\rtimes^\varphi  G})$, we first need to state the relation between $\Semi_{G,E}$ and $\Semi_{{E^*}\rtimes^\varphi  G}$. On one side, we have
$$\Semi_{G,E}=\{(\alpha, g, \beta) : \alpha,\beta\in E^*, g\in G, s(\alpha)=g\cdot s(\beta)\}.$$
On the other side,
$$\Semi_{{E^*}\rtimes^\varphi  G}=\langle \elmap{(\alpha, g)}{(\beta,h)} : \alpha,\beta\in E^*, g,h\in G, g^{-1}\cdot s(\alpha)=h^{-1}\cdot s(\beta)\rangle.$$
Since $(x,g)\in ({{E^*}\rtimes^\varphi  G})^{-1}$ for all $x\in E^0$, $g\in G$, we have that 
$$\elmap{(\alpha, g)}{(\beta,h)} =\elmap{(\alpha, g)}{(h^{-1}\cdot s(\beta),h^{-1})}\elmap{(s(\beta), h)}{(\beta, \ideG)}=\elmap{(\alpha, gh^{-1})}{(\beta,\ideG)}.$$
Moreover, since $E^*$ is singly aligned, then so is ${{E^*}\rtimes^\varphi  G}$ by \cite[Proposition 3.12(ii)]{BKQS}. Thus, by \cite[Theorem 3.2]{DGKMW}, 
$$\Semi_{{E^*}\rtimes^\varphi  G}=\{ \elmap{(\alpha, g)}{(\beta,\ideG)} : \alpha,\beta\in E^*, g\in G, s(\alpha)=g\cdot s(\beta)\}\,.$$
Hence, the map
$$
\begin{array}{cccc}
\pi: & \Semi_{G,E} &  \rightarrow &  \Semi_{{E^*}\rtimes^\varphi  G} \\
 & (\alpha, g, \beta) & \mapsto  &   \elmap{(\alpha, g)}{(\beta,\ideG)}
\end{array}
$$
is a well-defined, onto $\ast$-semigroup homomorphism. Let us characterize when $\pi$ is injective. To this end, take $(\alpha, g, \beta), (\gamma, h, \delta)\in \Semi_{G,E}$ such that
$$ \elmap{(\alpha, g)}{(\beta,\ideG)}=\pi(\alpha, g, \beta)=\pi(\gamma, h, \delta)= \elmap{(\gamma, h)}{(\delta,\ideG)}.$$
Being both equal functions, they must have the same domain, i.e. $(\beta, \ideG)({E^*}\rtimes^\varphi  G)=(\delta, \ideG)({E^*}\rtimes^\varphi  G)$. Since $({{E^*}\rtimes^\varphi  G})^{-1}=E^0\times G$, we conclude that $\beta=\delta$. Moreover, $\tau^{(\alpha, g)}=\tau^{(\gamma, h)}$ on their common domain, so that for every $\lambda \in (g^{-1}\cdot s(\alpha))E^*$ and for every $\ell \in G$ we have
$$(\alpha(g\cdot \lambda), \varphi(g,\lambda)\ell)=\tau^{(\alpha, g)}(\lambda, \ell)=\tau^{(\gamma, h)}(\lambda, \ell)=(\gamma(h\cdot\lambda), \varphi(h,\lambda)\ell).$$
Since the self-similar action of $G$ on $E^*$ preserves lengths of paths, we conclude that $\alpha=\gamma$, and that for every $\lambda \in (g^{-1}\cdot s(\alpha))E^*$ we have $g\cdot \lambda=h\cdot\lambda$ and $\varphi(g,\lambda)=\varphi(h,\lambda)$. Thus, the existence of nontrivial kernel for $\pi$ is equivalent to the existence of $g\in G$, $\alpha \in E^*$ such that for all $\lambda \in s(\alpha)E^*$ satisfies $g\cdot\lambda=\lambda$ and $\varphi(g,\lambda)=\ideG$; in other words, injectivity of $\pi$ is equivalent to the fact that the self-similar action of $G$ on $E^*$ is faithful on vertex-based trees of $E$. Notice that if $(E,G, \varphi)$ is pseudo free, then the above condition is trivially fulfilled, so that $\pi$ will be an isomorphism in this case. Moreover, being ${{E^*}\rtimes^\varphi  G}$ singly aligned, we have that it is right cancellative exactly when $(E,G, \varphi)$ is pseudo free. In this case, not only $\Semi_{G,E}\cong \Semi_{{E^*}\rtimes^\varphi  G}$, but also they are weak semilattices by Proposition \ref{propo1_4_2}, so that their associated tight groupoids are Hausdorff by Corollary \ref{corollary3_4_2}.

Now, we proceed to look at the relation between the corresponding tight groupoids $\G_{tight}(\Semi_{G,E})$ and $\G_{tight}(\Semi_{{E^*}\rtimes^\varphi  G})$. First, notice that the idempotent semilattices of $\Semi_E$, $\Semi_{G,E}$ and $\Semi_{{E^*}\rtimes^\varphi  G}$ coincide, so that the spaces of filters, ultrafilters and tight filters are the same (up to natural isomorphism). Additionally, the partial actions $\Semi_{G,E}\curvearrowright \hat{\mathcal{E}}_0$ and $\Semi_{{E^*}\rtimes^\varphi  G}\curvearrowright \hat{\mathcal{E}}_0$ are $\pi$-equivariant; also, the germ relation is compatible with $\pi$. Thus, $\pi$ induces a continuous, open, onto groupoid homomorphism
$$
\begin{array}{cccc}
\Phi: & \G_{tight}(\Semi_{G,E}) & \rightarrow  & \G_{tight}(\Semi_{{E^*}\rtimes^\varphi  G})  \\
 & [\alpha, g, \beta; \eta] &  \mapsto &   [\elmap{(\alpha, g)}{(\beta, \ideG)}; \eta] .
\end{array}
$$
Using the Morita equivalence reduction \cite[Theorem 3.2]{EPS}, we can assume that $E$ is row-finite with no sources or sinks, whence $\hat{\mathcal{E}}_\infty=\hat{\mathcal{E}}_{\text{tight}}=E^\infty$; let us reduce to this case, in order to simplify the computations. We now will show that $\Phi$ is injective. To this end, let $[\alpha, g, \beta; \eta] \in \ker \Phi$. This means that $\elmap{(\alpha, g)}{(\beta, \ideG)}$ is an idempotent. According to the computations done before, this happens exactly when $\alpha=\beta$ and for every $\lambda \in s(\alpha)E^*$ we have that $g\cdot\lambda=\lambda$ and $\varphi(g,\lambda)=\ideG$. Pick $\lambda$ any initial segment in $\eta\in E^\infty$. Notice that $\lambda \in s(\alpha)E^*$, and thus $(\alpha\lambda, \ideG, \alpha\lambda)\in \alpha\eta$ (seen as a filter), while
$$(\alpha, g, \alpha)\cdot (\alpha\lambda, \ideG, \alpha\lambda)= (\alpha(g\lambda), \varphi (g, \lambda), \alpha\lambda)=(\alpha\lambda, \ideG, \alpha\lambda).$$
Hence, by the germ relation, if $\eta=\lambda\hat{\eta}$, then 
$$[\alpha, g, \alpha; \alpha\eta)]=[\alpha\lambda, \ideG, \alpha\lambda; \alpha\lambda\hat{\eta})]\in \G_{tight}(\Semi_{G,E})^{(0)}.$$
Thus, $\Phi$ is a homeomorphism and an isomorphism of groupoids. This guarantees that, independently of the choice for representing the self-similar graph system $(G,E,\varphi)$, their associated tight groupoids -and hence their algebras- are the same.

\section{Amenability}

Now we are going to study a case where we can deduce amenability of $\G_{tight}(\Semi_{\Lambda\rtimes^\varphi  G})$ assuming that $\G_{tight}(\Semi_{\Lambda})$ and $G$ are amenable. Let $\Lambda$ be a finitely aligned LCSC, and let $\Gamma$ be a subsemigroup of a group $Q$.  

\begin{definition}[{\cite[Definition 6.1]{RW}}]
Let $\Gamma$ be a semigroup with unit element $\ideQ$. A \emph{$\Gamma$-graph} is a LCSC  $\Lambda$ together with a map, called the degree map, $\dmap:\Lambda\to\Gamma$, such that:
\begin{enumerate}
	\item $\dmap(\alpha\beta)=\dmap(\alpha)\dmap(\beta)$ for every $\alpha,\beta\in\Lambda$ with $s(\alpha)=r(\beta)$,
	\item for every $\alpha\in \Lambda$ and $\gamma_1,\gamma_2\in \Gamma$ with $\dmap(\alpha)=\gamma_1\gamma_2$, there are unique $\alpha_1,\alpha_2\in \Lambda$ with $s(\alpha_1)=r(\alpha_2)$, $\dmap(\alpha_i)=\gamma_i$ for $i=1,2$, such that $\alpha=\alpha_1\alpha_2$ (\emph{unique factorization property}).
\end{enumerate}	 

\end{definition}

 Observe that if $\Lambda$ is a $\Gamma$-graph, the unique factorization property   implies  that $\Lambda$ is right and left cancellative category, and does not have inverses.

Given two elements $\gamma_1,\gamma_2\in \Gamma$ 
$$\gamma_1\leq \gamma_2\qquad\text{if and only if}\qquad \gamma_1^{-1}\gamma_2 \in\Gamma\,.$$

\begin{lemma}\label{lemma_equality}
	Let  $\Lambda$ be a $\Gamma$-graph. Let $\alpha,\beta\in \Lambda $ with $ \alpha\Cap \beta$. Then  $\alpha\leq \beta$ if and only if $\dmap(\alpha)\leq \dmap(\beta)$. In particular $\alpha=\beta$ whenever $\dmap(\alpha)=\dmap(\beta)$.
\end{lemma}
\begin{proof}
	Let $\alpha,\beta\in \Lambda$, and let  $\varepsilon\in \alpha\vee \beta$, so there are $\delta,\eta\in \Lambda$ such that $\varepsilon=\alpha\delta=\beta\eta$. 
	Assume that $\dmap(\alpha)\leq \dmap(\beta)$. So by the unique factorization property there exists $\gamma,\gamma'\in \Lambda$ such that $\gamma\gamma'=\beta$ and $\dmap(\gamma)=\dmap(\alpha)$. But then 
	$$\dmap(\alpha)\dmap(\delta)=\dmap(\alpha\delta)=\dmap(\gamma\gamma'\eta)=\dmap(\gamma)\dmap(\gamma'\eta)\,.$$
	Thus, by the unique factorization property $\alpha=\gamma$, and hence $\alpha\leq \beta$, as desired.

	Finally, if $\dmap(\alpha)=\dmap(\beta)$ then $\alpha\leq \beta$ and $\beta\leq \alpha$. But since $\Lambda$ has no inverses, it follows that $\alpha=\beta$.
\end{proof}



\begin{definition}
	Let $(\Lambda,G,\varphi)$  be a category system, where $\Lambda$ is a $\Gamma$-graph. We say that $\Gamma$ is \emph{compatible}  with respect to $(\Lambda,G,\varphi)$, if $\dmap(g\cdot \alpha)=\dmap(\alpha)$ for every $g\in G$ and $\alpha\in\Lambda$ (\emph{G-invariant}).  
\end{definition}
 
Let $\Lambda$ be a $\Gamma$-graph compatible with respect to $(\Lambda,G,\varphi)$. Observe that since $\elmap{(\alpha,g)}{(\alpha,g)}=\elmap{(\alpha,\ideG)}{(\alpha,\ideG)}$ for every $(\alpha,g)\in \Lambda\rtimes^\varphi G $,
the set of idempotents  $\Idemp_\Lambda=\Idemp_{\Lambda\rtimes^\varphi G}$ coincide, and hence so does their spaces of tight filters. We will denote by $\FiltT$ the space of tight filters of $\Idemp_\Lambda$ and $\Idemp_{\Lambda\rtimes^\varphi G}$. By Lemma \ref{lemma_compresion},
$$\G_{tight}(\Semi_{\Lambda}) =\{[\elmap{\alpha}{\beta},\xi]: \xi\in \FiltT,\, \alpha,\beta\in\Lambda,\,s(\alpha)=s(\beta),\,\beta\in \Delta_\xi \} 
$$
and
$$
\G_{tight}(\Semi_{\Lambda\rtimes^\varphi G})\hspace{12truecm}
$$
$$
=\{[\elmap{(\alpha,a)}{(\beta,b)},\xi]: \xi\in \FiltT,\,\alpha,\beta\in\Lambda,\,a,b\in G,\,a^{-1}\cdot s(\alpha)=b^{-1}\cdot s(\beta) ,\,\beta\in \Delta_\xi\}\,.
$$

Then, we will think of $\G_{tight}(\Semi_{\Lambda})$ as an open subgroupoid of $\G_{tight}(\Semi_{\Lambda\rtimes^\varphi G})$ via the map $[\elmap{\alpha}{\beta},\xi]\mapsto [\elmap{(\alpha,\ideG)}{(\beta,\ideG)},\xi]$.

The following remark is going to be used repeatedly during the rest of the paper sometimes without mention it. 
\begin{remark}\label{remark_simpli}
 Let $(\alpha,a),(\beta,b)\in \Lambda\rtimes^\varphi G$,  and suppose that $(\alpha,a)\leq(\beta,b)$, then  there exists $(\delta,d)\in \Lambda\rtimes^\varphi G $ with  $r(\delta,d)=s(\alpha,a)$, that is, $r(\delta)=a^{-1}\cdot s(\alpha)$ such that 
 $$(\beta,b)=(\alpha,a)(\delta,d)=(\alpha(a\cdot \delta),\varphi(a,\delta)d)\,.$$
 Whence  $\alpha(a\cdot \delta)= \beta$ and $b=\varphi(a,\delta)d$. Hence, $a\cdot \delta=\sigma^\alpha(\beta)$ by left cancellation, so $\delta=a^{-1}\cdot \sigma^\alpha(\beta)$ and $d=\varphi(a,a^{-1}\cdot\sigma^\alpha(\beta))^{-1}b=\varphi(a^{-1},\sigma^\alpha(\beta))b$ because of the cocycle identity. Therefore,  
 $(\alpha,a)\leq(\beta,b)$ if and only if $\alpha\leq \beta$, and then we have that 
 $$\sigma^{(\alpha,a)}(\beta,b)=(a^{-1}\cdot \sigma^\alpha(\beta),\varphi(a^{-1},\sigma^\alpha(\beta))b)\,.$$
\end{remark}

\begin{lemma}\label{cocyle}
Let $\Lambda$ be a $\Gamma$-graph compatible with respect to $(\Lambda,G,\varphi)$, then the map
$$\dbar:\G_{tight}(\Semi_{\Lambda\rtimes^\varphi G})\to Q\,,\qquad [\elmap{(\alpha,a)}{(\beta,b)},\xi]\mapsto \dmap(\alpha)\dmap(\beta)^{-1}\,,$$
is a well defined continuous groupoid homomorphism. In particular, $\dbar$ restricts to $\G_{tight}(\Semi_{\Lambda})$.
\end{lemma}
\begin{proof}
	First we will prove that $\dbar$ is well defined. Let $[s,\xi]=[t,\xi]$ in $\G_{tight}(\Semi_{\Lambda\rtimes^\varphi G})$, that is, there exists $f\in \xi$ such that $sf=tf$. Without lost of generality we can assume that $s=\elmap{(\alpha,a)}{(\beta,b)}$, $t=\elmap{(\delta,d)}{(\eta,c)}$, and $f=\elmap{(\gamma,\ideG)}{(\gamma,\ideG)}$ with $\alpha,\beta,\gamma\in\Delta_\xi$ such that $\alpha,\beta\leq \gamma$.  Then, by Remark \ref{remark_simpli}, we have that  
	\begin{align*}
sf &= \elmap{(\alpha,a)}{(\beta,b)}\cdot\elmap{(\gamma,\ideG)}{(\gamma,\ideG)} \\ & =\elmap{(\alpha,a)(b^{-1}\cdot \sigma^\beta(\gamma),\varphi(b^{-1}, \sigma^\beta(\gamma)))}{(\gamma,\ideG)} \\
& =\elmap{(\alpha(ab^{-1}\cdot \sigma^\beta(\gamma)),\varphi(ab^{-1}, \sigma^\beta(\gamma)))}{(\gamma,\ideG)}\,,
	\end{align*}
and
\begin{align*}
tf& = \elmap{(\delta,d)}{(\eta,c)}\cdot\elmap{(\gamma,\ideG)}{(\gamma,\ideG)} \\ & =\elmap{(\delta,d)(c^{-1}\cdot \sigma^\eta(\gamma),\varphi(c^{-1}, \sigma^\eta(\gamma)))}{(\gamma,\ideG)} \\
& =\elmap{(\delta(dc^{-1}\cdot \sigma^\eta(\gamma)),\varphi(dc^{-1}, \sigma^\eta(\gamma)))}{(\gamma,\ideG)}\,,
	\end{align*}
	
	But $sf=tf$, so then
		\begin{equation}\label{important_eq}
\alpha(ab^{-1}\cdot \sigma^\beta(\gamma))=\delta(dc^{-1}\cdot \sigma^\eta(\gamma))\qquad\text{and}\qquad\varphi(ab^{-1}, \sigma^\beta(\gamma))=\varphi(dc^{-1}, \sigma^\eta(\gamma)) \,.
		\end{equation}
		
		Therefore, by the $G$-invariance of $\dmap$ we have that 
		\begin{align*}
\dmap(\alpha(ab^{-1}\cdot \sigma^\beta(\gamma)))\dmap(\gamma)^{-1}& =\dmap(\alpha)\dmap(ab^{-1}\cdot \sigma^\beta(\gamma))\dmap(\gamma)^{-1} \\ & =\dmap(\alpha\sigma^\beta(\gamma))\dmap(\gamma)^{-1} \\ & =\dmap(\alpha)\dmap(\beta)^{-1}\dmap(\gamma)\dmap(\gamma)^{-1}\\ & =\dmap(\alpha)\dmap(\beta)^{-1}\,,
		\end{align*}
		is equal to 
		\begin{align*}
\dmap(\delta(dc^{-1}\cdot \sigma^\eta(\varepsilon')))\dmap(\gamma\sigma^\gamma(\varepsilon')\theta)^{-1} &= \dmap(\delta)\d(dc^{-1}\cdot \sigma^\eta(\gamma)\theta)\dmap(\gamma)^{-1} \\ 
& =\dmap(\delta)\dmap(\sigma^\eta(\gamma))\dmap(\gamma)^{-1}\\ & =\dmap(\delta)\dmap(\eta)^{-1}\dmap(\gamma)\dmap(\gamma)^{-1}\\ & =\dmap(\delta)\dmap(\eta)^{-1}\,.
		\end{align*}
		Hence, $\dbar([\elmap{\alpha}{\beta},\xi])=\dmap(\alpha)\dmap(\beta)^{-1}=\dmap(\delta)\dmap(\eta)^{-1}=\dbar([\elmap{\delta}{\eta},\xi])$, so $\dbar$ is well-defined.

		Now, let $[s,\xi],[t,\xi']\in\G_{tight}(\Semi_{\Lambda\rtimes^\varphi G})$ such that $\xi=t\cdot \xi'$. If $s=\elmap{(\alpha,a)}{(\beta,b)}$ and $t=\elmap{(\delta,d)}{(\eta,c)}$, then 
		$[s,\xi]\cdot [t,\xi']=[st,\xi']$. Observe that, since $\beta,\delta\in \Delta_\xi$, by Lemma \ref{lemma_equality} there exists only one element $\varepsilon\in (\beta\vee \delta)\cap\Delta_\xi$. We define $f=\elmap{(\varepsilon,\ideG)}{(\varepsilon,\ideG)}$. We have that $\xi=t\cdot \xi'$, that means 
		\begin{align*}
\Delta_{\xi} & =\{\gamma\in \Lambda: \gamma\leq t(\nu),\,\nu\in \Delta_{\xi'} \}=\{\gamma\in \Lambda: \gamma\leq \elmap{(\delta,d)}{(\eta,c)}(\eta\nu),\,\eta\nu\in \Delta_{\xi'} \} \\
& = \{\gamma\in \Lambda: \gamma\leq \delta(dc^{-1}\cdot \nu),\,\eta\nu\in \Delta_{\xi'} \}\,.
		\end{align*}
		But $\varepsilon\in \Delta_\xi$, that is $\varepsilon\leq \delta(dc^{-1}\cdot \nu)$ for some $\eta\nu\in \Delta_{\xi'}$, and hence $$\dmap(\varepsilon)\leq \dmap(\delta(dc^{-1}\cdot \nu))=\dmap(\delta)\dmap(dc^{-1}\cdot \nu)=\dmap(\delta)\dmap(\nu)\,.$$
		As  $\dmap(\delta)\leq \dmap(\varepsilon)$, we have $\dmap(\delta)\leq \dmap(\varepsilon)\leq \dmap(\delta)\dmap(\nu)$.
	Therefore, there exists $g,h\in\Gamma$ such that $\dmap(\varepsilon)=\dmap(\delta)g$ and $gh=\dmap(\nu)$. Now, by the unique factorization, there exist unique elements $\nu_1,\nu_2\in \Lambda$ such that $\nu_1\nu_2=\nu$ and $\dmap(\nu_1)=g$ and $\dmap(\nu_2)=h$. But $\eta\nu_1\in\Delta_{\xi'}$, so $t(\eta\nu_1)=\delta(dc^{-1}\cdot\nu_1)$, and $\dmap(\delta(dc^{-1}\cdot \nu_1))=\dmap(\delta)g=\dmap(\varepsilon)$. Hence by Lemma \ref{lemma_equality} we have that $\varepsilon=\delta(dc^{-1}\cdot \nu_1)$. 
	
	Then $ft=t\cdot \elmap{(\eta\nu_1,\ideG)}{(\eta\nu_1,\ideG)}$, so  $[s,\xi]=[sf,\xi]$ and $[t,\xi']=[ft,\xi']$. So, we can assume that $\beta=\delta$, and hence $\dbar([s,\xi])=\dmap(\alpha)\dmap(\beta)^{-1}$ and $\dbar([t,\xi'])=\dmap(\beta)\dmap(\eta)^{-1}$. Thus, 
		\begin{align*}
\elmap{(\alpha,a)}{(\beta,b)}\cdot \elmap{(\beta,d)}{(\eta,c)} & =\elmap{(\alpha,a)(b^{-1}\cdot s(\beta),b^{-1}d)}{(\eta,c)} \\
& = \elmap{(\alpha,ab^{-1}d)}{(\eta,c)} \,.
		\end{align*} 
		Therefore, 
		\begin{align*}\label{eq_mult}
\dbar([st,\xi]) & =\dmap(\alpha)\dmap(\eta)^{-1} \\
& =\dmap(\alpha)\dmap(\beta)^{-1}\dmap(\beta)\dmap(\eta)^{-1} \\
 & =\dbar([s,\xi])\dbar([t,\xi'])\,.
		\end{align*}
		Thus, $\dbar$ is a morphism of groupoids.
		
		Finally, given $g\in \Gamma$, we have that $$(\dbar)^{-1}(g)=\bigcup_{\alpha,\beta\in\Lambda,\,\dmap(\alpha)\dmap(\beta)^{-1}=g}\Theta(\elmap{(\alpha,a)}{(\beta,b)}, D_{\elmap{\beta}{\beta}})\,,$$ that is an open. Thus, $\dbar$ is continuous. 
\end{proof}

Let $\dbar:\G_{tight}(\Semi_{\Lambda\rtimes^\varphi G})\to \Gamma$ be the cocycle defined in Lemma \ref{cocyle}, and let us define
$$\HG_{\Lambda\rtimes^\varphi G}:=(\dbar)^{-1}(\ideG)=\{[\elmap{(\alpha,a)}{(\beta,b)},\xi]\in\G_{tight}(\Semi_{\Lambda\rtimes^\varphi G}) : \dmap(\alpha)\dmap(\beta)^{-1}=\ideG \}\,.$$ It is an open subgroupoid of $\G_{tight}(\Semi_{\Lambda\rtimes^\varphi G})$.

Now in order to be able to decompose the groupoid $\HG_{\Lambda\rtimes^\varphi G}$ as a union of more treatable groupoids, we need to impose some conditions on the semigroup $\Gamma$.  

\begin{definition}
	Let $\Gamma\subseteq Q$ be a subsemigroup of a group $Q$ with $\Gamma\cap \Gamma^{-1}=\ideQ$. We say that $\Gamma$ is a \emph{join-semilattice} if given $g_1,g_2\in \Gamma$ 
	$$\inf\{g\in \Gamma: g_1,g_2\leq g\}$$
	exists an it is unique. We will denote it by $g_1\vee g_2$.
\end{definition}

We now assume that $\Gamma$ is a join-semilattice. Then, given $g\in \Gamma$, we define
$$\HG_{\Lambda\rtimes^\varphi G}^{(g)}:=\{[\elmap{(\alpha,a)}{(\beta,b)},\xi]: \dmap(\alpha)=\dmap(\beta)\leq g\}\,.$$
We claim that $\HG_{\Lambda\rtimes^\varphi G}^{(g)}$ is an open subgroupoid of $\HG_{\Lambda\rtimes^\varphi G}$. Let  $[\elmap{(\alpha,a)}{(\beta,b)},\xi], [\elmap{(\delta,d)}{(\eta,c)},\xi']\in \HG^{(g)}_{\Lambda\rtimes^\varphi G}$ two composable elements with $g_1:=\dmap(\beta)=\dmap(\alpha)\leq g$ and $g_2:=\dmap(\delta)=\dmap(\eta)\leq g$. Since $\beta,\delta\in\Delta_\xi$ we have that there exists $\varepsilon\in (\beta\vee \delta)\cap \Delta_{\xi}$, and $g_1,g_2\leq \dmap(\varepsilon)$. Since $\Gamma$ is a join-semilattice we have that $g_1\vee g_2\leq \dmap(\varepsilon)$. Then by the unique factorization property there exists $\varepsilon_1,\varepsilon_2\in \Lambda$ with $\dmap(\varepsilon_1)=g_1\vee g_2$ and $\varepsilon=\varepsilon_1\varepsilon_2$.  Then $\varepsilon_1\in\Delta_\xi$ and hence by Lemma 
\ref{lemma_equality} we have that $\beta,\delta\leq \varepsilon_1$. Then as shown in the proof of Lemma \ref{cocyle} we can find elements $[\elmap{(\alpha',a')}{(\varepsilon,b')},\xi], [\elmap{(\varepsilon,d')}{(\eta',c')},\xi']\in \HG_{\Lambda\rtimes^\varphi G}$ with $[\elmap{(\alpha,a)}{(\beta,b)},\xi]=[\elmap{(\alpha',a')}{(\varepsilon,b')},\xi]$ and $[\elmap{(\delta,d)}{(\eta,c)},\xi']=[\elmap{(\varepsilon,d')}{(\eta',c')},\xi']$, and the product 
$$[\elmap{(\alpha',a')}{(\varepsilon,b')},\xi]\cdot [\elmap{(\varepsilon,d')}{(\eta',c')},\xi']=[\elmap{(\alpha',a'(b')^{-1})}{(\eta',c')},\xi']\in \HG_{\Lambda\rtimes^\varphi G}^{(g_1\vee g_2)}\,.$$ 
Therefore,   $\HG_{\Lambda\rtimes^\varphi G}^{(g_1\vee g_2)}\subseteq \HG_{\Lambda\rtimes^\varphi G}^{(g)}$, as desired.

Moreover, as a consequence of the above computation, given $g_1,g_2\leq g$ we have that $\HG_{\Lambda\rtimes^\varphi G}^{(g_1)}\HG_{\Lambda\rtimes^\varphi G}^{(g_2)}\subseteq \HG_{\Lambda\rtimes^\varphi G}^{(g)}$.  Then, if $\Gamma$ is countable, there exists an ascending sequence of elements $g_1,g_2,\ldots \in \Gamma$ such that for every $g\in \Gamma$ there exists $n\in \NN$ with $g\leq g_n$. Whence,   $\HG_{\Lambda\rtimes^\varphi G}=\bigcup_{i=1}^{\infty}\HG_{\Lambda\rtimes^\varphi G}^{(g_i)}$.

The next step will be to define a cocycle of the groupoids $\HG_{\Lambda\rtimes^\varphi G}^{(g)}$ onto $G$. In order to do that we will need to make the following assumption in the $\Gamma$-graph $\Lambda$.

\begin{definition}\label{def_bigstar}
Let $\Lambda$ be a $\Gamma$-graph. Then  $\Lambda$ satisfies property $(\bigstar)$ if given $F\in \Delta_{tight}$ and $g\in \Gamma$, then there exists a unique $\beta\in F$ with $\dmap(\beta)\leq g$ such that whenever $\alpha\in F$ satisfies $\dmap(\alpha)\leq g$,  we have that $\alpha\leq \beta$. 
\end{definition}

We can give some condition on $\Gamma$ to guarantee that every $\Gamma$-graph satisfies condition $(\bigstar)$.

\begin{proposition}
	Let $\Lambda$ be a $\Gamma$-graph, and assume every bounded ascending sequence of elements of $\Gamma$ stabilizes. Then  $\Lambda$ satisfies property $(\bigstar)$.
\end{proposition}
\begin{proof}
	We define $F_g:=\{\beta\in F: \dmap(\beta)\leq g \}$. Observe that given $\alpha,\beta\in F_g$ with $\dmap(\beta)\leq \dmap(\alpha)$, then $\dmap(\alpha\vee \beta)=\dmap(\alpha)\vee\dmap(\beta)=\dmap(\alpha)\leq g$, hence $\alpha\vee\beta\in F_g$, and the unique factorization property says that $\alpha=\alpha\vee\beta$, and hence $\beta\leq \alpha$. So it is enough to prove that there exists $\alpha\in F_g$ such that $\dmap(\beta)\leq \dmap(\alpha)$ for every $\beta\in F_g$.

	Let $\alpha_0\in F$, and let $\beta\in F_d$ with $\dmap(\beta)\nleq \dmap(\alpha_0)$. If such $\beta$ does not exists, then we are done. Otherwise,  $\dmap(\alpha_0\vee\beta)=\dmap(\alpha_0)\vee \dmap(\beta)\leq g$, and hence $\alpha_0\vee\beta\in F_g$, with $\dmap(\alpha_0)<\dmap(\alpha_0\vee\beta)$, because if $\dmap(\alpha_0)=\dmap(\alpha_0\vee\beta)$ then $\alpha_0=\alpha_0\vee\beta$ by Lemma \ref{lemma_equality}. Let us define $\alpha_1:=\alpha_0\vee \beta$, so
	 $\dmap(\alpha_0)<\dmap(\alpha_1)$. Now if in this way we could construct an infinite sequence $\alpha_0,\alpha_1,\alpha_2,\ldots\in F_g$ such that $\d(\alpha_i)<\dmap(\alpha_{i+1})$, then this will contradict the hypothesis. Then will be $n$ such that $\dmap(\beta)\leq \dmap(\alpha_n)$ for every $\beta\in F_g$, and so $\alpha:=\alpha_n$, and we are done. 
	  
\end{proof}

\begin{example}
	Every $\mathbb{N}^k$-graph $\Lambda$  satisfies property $(\bigstar)$.
\end{example}

Now we are ready to define the promised cocyle.

\begin{proposition}\label{G_cocycle}
	Let $\Lambda$ be a $\Gamma$-graph compatible with respect to a pseudo free system $(\Lambda,G,\varphi)$, and suppose that $\Lambda$ satisfies property $(\bigstar)$. Then for every $g\in\Gamma$ there exists a continuous groupoid homomorphism
	$$\tmap^{(g)}:\HG_{\Lambda\rtimes^\varphi G}^{(g)}\to G\,.$$
\end{proposition}
\begin{proof}
	Let $[s,\xi]\in \HG_{\Lambda\rtimes^\varphi G}^{(g)}$. By property $(\bigstar)$ there exists $\beta\in \Delta_\xi$ such that $\delta\leq\beta$ for every $\delta\in\Delta_\xi$ with $\dmap(\delta)\leq g$. If we define $f=\elmap{(\beta,e)}{(\beta,e)}$, then we have that $[s,\xi]=[sf,\xi]$ whenever $s=\elmap{(\alpha,a)}{(\delta,b)}$ with $\dmap(\delta)\leq g$, and 
	\begin{align*}
sf=\elmap{(\alpha,a)}{(\delta,b)}\elmap{(\beta,e)}{(\beta,e)} & =\elmap{(\alpha,a)(b^{-1}\cdot \sigma^{\delta}(\beta)),\varphi(b^{-1},\sigma^\delta(\beta)))}{(\beta,e)} \\
& = \elmap{(\alpha(ab^{-1}\cdot \sigma^{\delta}(\beta))),\varphi(ab^{-1},\sigma^\delta(\beta)))}{(\beta,e)}\,.
	\end{align*}
	Thus, without lost of generality,  any element $[s, \xi]\in \HG_{\Lambda\rtimes^\varphi G}^{(g)}$ has a representative of the form $[\elmap{(\alpha,a)}{(\beta,b)},\xi]$, where $\beta$ is the unique maximal element in $\Delta_\xi$ satisfying $\dmap(\beta)\leq g$ given by property $(\bigstar)$.
	
	Under this choice of representative, we define $\tmap^{(g)}:\HG_{\Lambda\rtimes^\varphi G}^{(g)}\to G$ by the rule 
	$$\tmap^{(g)}([\elmap{(\alpha,a)}{(\beta,b)},\xi])=ab^{-1}.$$
	
	 Let us check that $\tmap^{(g)}$ is well defined. To this end, let $[s,\xi]$ and $[s',\xi]$ in $\HG_{\Lambda\rtimes^\varphi G}^{(g)}$ with $[s,\xi]=[s',\xi]$. By the above argument, we can assume that $s=\elmap{(\alpha,a)}{(\beta,b)}$ and $s'=\elmap{(\alpha',a')}{(\beta,b')}$. 
	Let $h=\elmap{(\beta\beta',\ideG)}{(\beta\beta',\ideG)}$ for some $\beta'$ such that $\beta\beta'\in \Delta_\xi$ and $sh=s'h$. Then,
$$sh =\elmap{(\alpha(ab^{-1}\cdot\beta'),\varphi(ab^{-1}, \beta'))}{(\beta\beta',\ideG)} \text{ and } s'h	  =\elmap{(\alpha(a'b'^{-1}\cdot \beta'),\varphi(a'b'^{-1}, \beta'))}{(\beta\beta',\ideG)}\,. 
$$
	Therefore we have that
	$$\alpha(ab^{-1}\cdot \beta')=\alpha(a'b'^{-1}\cdot \beta')\qquad\text{and}\qquad \varphi(ab^{-1},\beta')=\varphi(a'b'^{-1}, \beta')\,,$$
	and by left cancellation we have that 
	$$ab^{-1}\cdot \beta'=a'b'^{-1}\cdot \beta'\qquad\text{and}\qquad \varphi(ab^{-1}, \beta')=\varphi(a'b'^{-1}, \beta')\,.$$
	  Hence, by Proposition \ref{prop_pseudo},  $ab^{-1}=a'b'^{-1}$. Thus, $\tmap^{(g)}$ is well-defined.
	  
	  Now, given a composable pair $[s,\xi], [t,\xi']$, we can choose representatives $s=[\elmap{(\alpha,a)}{(\beta,b)},\xi]$ and $t=[\elmap{(\gamma,c)}{(\beta',b')},\xi']$ with $\beta, \beta'$ unique maximal elements in $\Delta_\xi , \Delta_{\xi'}$ (respectively) satisfying $\dmap(\beta), \dmap(\beta')\leq g$ given by property $(\bigstar)$. Since $\xi=t\cdot \xi'$, we have that $\gamma\in \Delta_\xi$, whence $\gamma=\beta\vee \gamma$ by property $(\bigstar)$. Thus, the computation performed in Lemma \ref{cocyle} do not require replace $\beta$ by any element $\delta$ with $\dmap (\delta)> g$, and thus this argument shows that $\tmap^{(g)}$ is a continuous groupoid homomorphism.
	  
\end{proof}

\begin{proposition}
	Let $\Lambda$ be a $\Gamma$-graph compatible with respect to a pseudo free system $(\Lambda,G,\varphi)$, and suppose that $\Lambda$ satisfies property $(\bigstar)$, with $G$ and $Q$ countable amenable groups. Moreover, assume that $\Gamma$ is a join-semilattice.  If the kernel of the map $\dbar:\G_{tight}(\Semi_\Lambda)\to Q$ is amenable, then $\G_{tight}(\Semi_{\Lambda\rtimes^\varphi G})$ is amenable. 
\end{proposition}
\begin{proof}
By \cite[Corollary 4.5]{RW} it is enough to prove that $\HG_{\Lambda\rtimes^\varphi G}$ is an amenable groupoid. As observe above $\HG_{\Lambda\rtimes^\varphi G}=\bigcup_{n=1}^\infty \HG_{\Lambda\rtimes^\varphi G}^{(g_n)}$ where $\HG_{\Lambda\rtimes^\varphi G}^{(g_n)}$ are open subgroupoids with $\HG_{\Lambda\rtimes^\varphi G}^{(g_n)}\subseteq \HG_{\Lambda\rtimes^\varphi G}^{(g_{n+1})}$ and $(\HG_{\Lambda\rtimes^\varphi G}^{(g_n)})^{(0)}= (\HG_{\Lambda\rtimes^\varphi G}^{(g_{n+1})})^{(0)}$. Then by  \cite[Section 5.2(c)]{AR} it is enough to prove that the groupoids  $\HG_{\Lambda\rtimes^\varphi G}^{(g_n)}$ are amenable for every $n$. But now
\begin{align*}
(\tmap^{(g_n)})^{-1}(\ideG) & =\{[\elmap{(\alpha,a)}{(\beta,a)},\xi]\in\HG_{\Lambda\rtimes^\varphi G}^{(g_n)}\} \\
& = \{[\elmap{(\alpha,\ideG)}{(\beta,\ideG)},\xi]\in\HG_{\Lambda\rtimes^\varphi G}^{(g_n)}\} \\
& = \{[\elmap{(\alpha,\ideG)}{(\beta,\ideG)},\xi]\in\G_{\Lambda\rtimes^\varphi G}: \dmap(\alpha)=\dmap(\beta)\leq g_n\}  \\
& = \{[\elmap{\alpha}{\beta},\xi]\in\G_{\Lambda\rtimes^\varphi G}: \dmap(\alpha)=\dmap(\beta)\leq g_n\}  \subseteq (\dbar)^{-1}(\ideQ)\,.
\end{align*}	
Therefore since $(\dbar)^{-1}(\ideQ)$ is amenable by assumption, then $(\tmap^{(g_n)})^{-1}(\ideG)$ is amenable. So using again \cite[Corollary 4.5]{RW} we have that $\HG_{\Lambda\rtimes^\varphi G}^{(g_n)}$ is amenable, as desired.
\end{proof}

Next step will prove to prove that the kernel of the map $\dbar:\G_{tight}(\Semi_\Lambda)\to Q$ is amenable. In order to do that we will prove that the groupoid $\G_{tight}(\Semi_\Lambda)$ is isomorphic to  the semigroup action groupoid of the $\Gamma$-graph $\Lambda$ defined in \cite[Section 5]{RW}. This semigroup action groupoid has also a canonical cocyle $\bar{\mathbf{c}}$ onto $Q$ which kernel is amenable. Now we will prove that the kernel of $\bar{\mathbf{c}}$ is isomorphic to the kernel of $\dbar$. First we introduce the semigroup action groupoid.

\begin{definition}
	Let $X$ be a set and $\Gamma\subseteq Q$ be a semigroup of a group $Q$ containing the identity  $\ideQ$. A \emph{left action of $\Gamma$ on $X$} consists of a subset $\Gamma\star X$ of $\Gamma\times X$ and a map $T:\Gamma\star X\to X$ sending $(g,x)\mapsto g\cdot x$, such that:
	\begin{enumerate}
		\item for all $x\in X$, $(e,x)\in \Gamma\star X$ and $e\cdot x=x$;
		\item for all $(g,h,x)\in \Gamma\times\Gamma\times X$, $(gh,x)\in \Gamma\star X$ if and only if  $(h,x)\in \Gamma\star X$ and $(g,h\cdot x)\in \Gamma\star X$, if this holds, $g\cdot (h\cdot x)=(gh)\cdot x$.
	\end{enumerate}
For all $g\in\Gamma$, we define $U(g):=\{x: (g,x)\in \Gamma\star X\}$ and $V(g)=\{g\cdot x: (g,x)\in\Gamma\star X \}$ and $T_g:U(g)\to V(g)$ the map such that $T_g(x)=g\cdot x$. The triple $(X,\Gamma,T)$ is called a \emph{semigroup action}.
\end{definition}

\begin{definition}
A semigroup action $(X,\Gamma,T)$ is called \emph{directed} if for all $g,h\in \Gamma$ such that $U(g)\cap U(h)\neq\emptyset$ there exists $r\in \Gamma$ with $g,h\leq r$ such that $U(g)\cap U(h)=U(r)$.
\end{definition}

When  $(X,\Gamma,T)$ is a directed semigroup action it is defined the groupoid 
$$\G(X,\Gamma,T)=\{(x,gh^{-1},y)\in X\times \Gamma \times X:\,(g,x),(h,y)\in \Gamma\star X,\, g\cdot x=h\cdot y\}\,.$$
Let $X$ be a locally compact Hausdorff space such that $U(g)$ and $V(g)$ are open subsets of $X$ for every $g\in \Gamma$, and $T_g:U(g)\to V(g)$ is a local homeomorphism. 

Given $g,h\in\Gamma$, $A,B$ subsets of $X$, we define
$$Z(A,g,h,B):=\{(x,gh^{-1},y)\in\G(X,\Gamma,T):x\in A,\,y\in B\text{ and }g\cdot x=h\cdot y\}\,.$$
The family  $\mathcal{B}$ of subsets $Z(A,g,h,B)$, with $A\subseteq U(g)$ and $B\subseteq U(h)$ open subsets, such that  $(T_g)_{|A}$ and $(T_h)_{|B}$ are injective, and $T_g(A)=T_h(B)$, forms a basis for the topology of $\G(X,\Gamma,T)$.  With this topology $\G(X,\Gamma,T)$ is a locally compact étale groupoid.

Now recall by Lemma \ref{lemma_equality}, that given $g\in \Gamma$ and $F\in \Delta^*$, $F\cap \dmap^{-1}(g)$ is either empty or contains just one element. Then we can define the following.

\begin{definition}
 	Let $\Lambda$ be a $\Gamma$-graph. Then 
 	$$\Gamma\star\FiltT=\{(g,\xi)\in \Gamma\times \FiltT: \exists \alpha\in \Delta_\xi\text{ such that }\dmap(\alpha)=g \}\,,$$
 	and given $(g,\xi)\in \Gamma\star\FiltT$, we define $T_g(\xi)=\sigma^\alpha\cdot \xi$, where $\alpha\in \Delta_\xi$ with $\dmap(\alpha)=g$.
 	$$U(g):=\bigcup_{\alpha\in \dmap^{-1}(g)} D_{\elmap{\alpha}{\alpha}}\qquad\text{and}\qquad V(g)=\bigcup_{\alpha\in \dmap^{-1}(g)}\left( \bigcup_{\beta\in \Lambda s(\alpha)} D_{\elmap{\beta}{\beta}}\right) \,.$$
 \end{definition}
 
 Let $\Lambda$ be a $\Gamma$-graph with $\Gamma$ a join-semilattice,  because of the factorization property we have that  $\dmap(\alpha\vee\beta)=\dmap(\alpha)\vee\dmap(\beta)$, and hence 
 $$U(g)\cap U(h)=\bigcup_{\alpha\in \dmap^{-1}(g),\,\beta\in \dmap^{-1}(h),\,\varepsilon\in \alpha\vee \beta} D_{\elmap{\varepsilon}{\varepsilon}}= U(g\vee h)\,.$$
 Thus,  the semigroup action $(\FiltT, \Gamma, T)$ is directed.

 \begin{proposition}
 Let $\Lambda$ be a $\Gamma$-graph with $\Gamma$ being join-semilattice. Then the map 
 $$\Phi:\G_{tight}(\Semi_{\Lambda})\to \G(\FiltT, \Gamma, T)\,,\qquad [\elmap{\alpha}{\beta},\xi]\mapsto (\elmap{\alpha}{\beta}\cdot \xi,\dmap(\alpha)\dmap(\beta)^{-1},\xi)$$
 is an isomorphism of topological groupoids.
 \end{proposition}
 \begin{proof}
 	First observe that $\Phi$ is well-defined because of  Lemma \ref{cocyle}, and it is then clearly a groupoid homomorphism. That $\Phi$ is a bijection follows from the definition of the inverse $\Phi^{-1}:\G(\FiltT, \Gamma, T)\to \Phi:\G_{tight}(\Semi_{\Lambda})$ by $\Phi^{-1}(\xi,gh^{-1},\xi')=[\elmap{\alpha}{\beta},\xi]$, where $\beta$ is the unique element in $\Delta_\xi$ with $\dmap(\beta)=h$ and $\alpha$ is the unique element in $\Delta_{\xi'}$ with $\dmap(\alpha)=g$ given by Lemma \ref{lemma_equality}.

 	 Now observe that the sets of the form $Z(\elmap{\alpha}{\beta}(D_{\elmap{\beta}{\beta}}),g,h,D_{\elmap{\beta}{\beta}})$ for 	
 	 $g,h\in \Gamma$, $\beta\in \dmap^{-1}(h)$ and $\alpha\in \dmap^{-1}(g)\cap \Lambda s(\beta)$ forms a basis for the topology of $\G(\FiltT, \Gamma, T)$. But $\Phi^{-1}(Z(\elmap{\alpha}{\beta}(D_{\elmap{\beta}{\beta}}),g,h,D_{\elmap{\beta}{\beta}}))=[\elmap{\alpha}{\beta},D_{\elmap{\beta}{\beta}}]$, thus $\Phi$ is continuous and also open.  
 \end{proof}
 
 Given a directed semigroup action $(X,\Gamma,T)$, there exists a natural groupoid homomorphism $\bar{\mathbf{c}}:\G(\FiltT, \Gamma, T)\to \Gamma$ defined by $\bar{\mathbf{c}}(x,gh^{-1},yu)=gh^{-1}$ \cite[Proposition 5.12]{RW}, and it is clear that $\Phi$ intertwines $\bar{\mathbf{c}}$  and $\dbar$, that is, $\bar{\mathbf{c}}\circ \Phi= \dbar$. Therefore $\bar{\mathbf{c}}^{-1}(\ideQ)$ and $\dbar^{-1}(\ideQ)$ are isomorphic as topological groupoids. But in the proof of  \cite[Theorem 5.13]{RW} it is proved that $\bar{\mathbf{c}}^{-1}(\ideQ)$ is amenable, whence $\dbar^{-1}(\ideQ)$ is amenable too.

 \begin{theorem}\label{amenable_final}Let $\Lambda$ be a $\Gamma$-graph compatible with respect to a pseudo free system $(\Lambda,G,\varphi)$, and suppose that $\Lambda$ satisfies property $(\bigstar)$, with $G$ and $Q$ countable amenable groups. Moreover, assume that $\Gamma$ is a join-semilattice. Then $\G_{tight}(\Semi_{\Lambda\rtimes^\varphi G})$ is amenable.
 \end{theorem}	
 
\section*{Acknowledgments}

Part of this work was done during visits of the second author to the Institutt for Matematiske Fag, Norges Teknisk-Naturvitenskapelige Universitet (Trondheim, Norway), and during an stage of both authors to the ICMAT (Universidad Aut\'onoma de Madrid, Spain) as part of the Thematic Research Program: \emph{Operator Algebras, Groups and Applications to Quantum Information} in 2019, and the work was significantly supported by the research environment and facilities provided there. Both authors thank the centers for their kind hospitality.

\end{document}